\documentclass[12pt,twoside]{amsart}
\usepackage{amssymb,amsmath,amsfonts,amsthm}
\usepackage{mathrsfs}
\usepackage[X2,T1]{fontenc}
\usepackage[applemac]{inputenc}
\usepackage{cite}
\usepackage{latexsym}
\usepackage[dvips]{graphicx}
\usepackage{comment}
\usepackage{ulem}
\def\px{\langle x \rangle}
\def\pd{\langle D \rangle}
\def\Im{\mathop{\rm Im}\nolimits}
\def\Re{\mathop{\rm Re}\nolimits}

\def\Im{\mathop{\rm Im}\nolimits}
\def\Re{\mathop{\rm Re}\nolimits}

\def\R{\mathbb R}
\def\C{\mathbb C}
\def\N{\mathbb N}

\def\ds{\displaystyle}

\newcommand\dslash{d\llap {\raisebox{.9ex}{$\scriptstyle-\!$}}}
\usepackage{colortbl}

\newcommand{\beqsn}{\arraycolsep1.5pt\begin{eqnarray*}}
	\newcommand{\eeqsn}{\end{eqnarray*}\arraycolsep5pt}
\newcommand{\beqs}{\arraycolsep1.5pt\begin{eqnarray}}
	\newcommand{\eeqs}{\end{eqnarray}\arraycolsep5pt}

\newtheorem{Th}{Theorem}[section]
\newtheorem{Rem}[Th]{Remark}

\newtheorem{Lemma}[Th]{Lemma}

\newtheorem{Def}[Th]{Definition}
\newtheorem{Prop}[Th]{Proposition}

\def\px{\langle x \rangle}

\def\pd{\langle D \rangle}
\catcode`\@=11

\renewcommand{\section}%
{\setcounter{equation}{0}\@startsection {section}{1}{\z@}{-3.5ex plus -1ex
		minus -.2ex}{2.3ex plus .2ex}{\Large\bf}}

\title{Smoothing effect for higher order dispersive equations and applications to nonlinear initial value problems}

\author[1]{Alexandre Arias Junior $^1$}
\author[2]{Alessia Ascanelli $^2$}
\author[3]{Marco Cappiello $^3$}

\begin{document}
	
	
	\begin{abstract} 
In this paper we deal with the initial value problem related to a family of dispersive inhomogeneous evolution equations $Pu=f$ with variable coefficients belonging to the class of $p$-evolution equations, $p\geq 2$. We study the smoothing effect produced by some spatial decay assumptions on the imaginary part of the subleading  coefficient of the linear operator $P$. Then we apply this result to nonlinear problems with derivative nonlinearities obtaining existence and uniqueness of the solution in a suitable Sobolev class. The nonlinear equations considered include various equations describing wave propagation in inhomogeneous media such as KdV-type and Kawahara-type equations with variable coefficients.\end{abstract}

	\maketitle
\noindent \textit{2020 Mathematics Subject Classification}: Primary 35B65; Secondary 35G25, 35S05  \\
\noindent	\textit{Keywords and phrases}: Smoothing estimates, $p$-evolution equations, nonlinear initial value problems
	
	\markboth{\sc Smoothing effect for dispersive equations of higher order}{\sc A. Arias Junior, A.~Ascanelli, M. Cappiello}
	\section{Introduction and main results}
The aim of this paper is to study the smoothing properties of a family of dispersive inhomo\-ge\-ne\-ous evolution equations of the form  $Pu=f$, where $P$ stands for the differential operator
\begin{equation}\label{op}
	P(t,x,D_t,D_x)=D_t+a_p(t,x)D_x^p+\sum_{j=0}^{p-1} a_j(t,x)D_x^j, \qquad (t,x)\in [0,T]\times \R,
\end{equation} 
$D=\frac 1i \partial$. We assume $a_j\in C([0,T];{\mathcal B}^\infty)$ for $0\leq j\leq p$, where ${\mathcal B}^\infty={\mathcal B}^\infty(\R)$ is the space of complex-valued functions which are bounded on $\R$ together with all their derivatives and that $a_p$ is real valued. Operators of the form \eqref{op} are also known in the literature as $p$-evolution operators  and $a_p$ real-valued guarantees that the assumptions of Lax-Mizohata theorem are satisfied, cf. \cite[Theorem 1.2]{Mizohata1961}.
\\ \indent
The well-posedness in $L^2(\R), H^s(\R)$ and $H^\infty(\R):= \cap_{s \in \R}H^s(\R)$ of the related linear Cauchy problem
\beqs
\label{genCP}
\begin{cases}
	P(t,x,D_t,D_x)u(t,x)=f(t,x) & (t,x)\in[0,T]\times\R\cr
	u(0,x)=g(x) & x\in\R
\end{cases}
\eeqs
has been investigated  by several authors, cf. \cite{CR, Doi2, I1, I2, KB} in the case $p=2$ and \cite{ABZsuff, ABZnec, ABmoser} for a generic $p$. In general, when some of the coefficients $a_j, j=1,\ldots, p-1$ are complex-valued, some control on the imaginary part of these coefficients is necessary to obtain well-posedness. In \cite{ABZnec, I1} for instance it was proved that a necessary condition for well-posedness of \eqref{genCP} in $H^\infty(\R)$ is the existence of
constants $M,N>0$ such that:
\begin{equation}
	\label{CN2}
	\sup_{x\in\R}\min_{0\leq\tau\leq t\leq T}\int_{-\varrho}^\varrho 
	\Im \, a_{p-1}(t,x+p a_p(\tau)\theta)d\theta\leq M\log(1+\varrho)+N,\qquad 
	\forall \varrho>0.
\end{equation}
In papers treating sufficient conditions for well-posedness, it is usual to assume pointwise decay conditions on the coefficients for $|x|$ large. In \cite{ABZsuff, ABmoser}, the well-posedness in $H^\infty(\R)$ for $p$-evolution operators of the form \eqref{op} has been obtained using very precise decay assumptions on the  coefficients $a_j$. If the decay is strong enough, it is possible to obtain well-posedness in Sobolev spaces without loss of derivatives. \\
 \indent 
In addition, it has been known for many years that under suitable assumptions on the coefficients, dispersive equations enjoy smoothing properties. These have been completely established for homogeneous equations ($f=0$) with constant coefficients, see for instance \cite{ConstantinSaut1, ConstantinSaut2, Kato, KPV1, Sugimoto}. Roughly speaking, it is well known that the solution gains $(p-1)/2$ derivatives with respect to the initial datum $g$. 
In the papers \cite{KPV2, KPV3, KPV4} Kenig, Ponce and Vega studied the case of inhomogenous equations and the smoothing effect also with respect to the forcing term $f$ establishing that it consists of $p-1$ derivatives. 
\\
For equations with variable $x$-dependent coefficients, the largest part of existing results concern Schr\"odinger-type equations, see \cite{Doi1, Doi2, FS, KPV5, KPRV}. Whenever the leading coefficient $a_p$ depends on $x$, a non-trapping condition is required in order to obtain the smoothing effect. Roughly speaking, this condition corresponds to the fact the bicharacteristic curves of the operator $P$ are globally defined and escape from any compact subset of $\R^{2n}$ when the parameter becomes large enough, cf. \cite{CKS}. Recently, Federico and Tramontana \cite{FT} obtained smoothing estimates for linear third order equations ($p=3$) replacing the  non-trapping condition with a "smallness assumption" on the coefficients and applied them to obtain local existence results for KdV-type equations with variable coefficients. In fact, an interesting consequence of the smoothing effect which has been analyzed in many of the above mentioned references is the possibility of proving local (in time) existence and uniqueness results for nonlinear Cauchy problems with linear part given by an operator satisfying smoothing estimates.
Concerning higher order equations with $x$-dependent coefficients, the most relevant reference is the paper \cite{chihara} by Chihara where the author obtained smoothing estimates for a $p$-evolution operator of arbitrary order under a decay condition on the subprincipal symbol and the non-trapping condition.\\

Starting from this state of the art, in this paper we want to consider operators of the form \eqref{op} with arbitrary degree of evolution $p$ and obtain more refined smoothing estimates compared to \cite{chihara, FT} for the inhomogeneous Cauchy problem \eqref{genCP} imposing weaker decay conditions on the lower order terms. Then, we shall apply this to obtain local existence and uniqueness of solutions for a nonlinear Cauchy problem with right hand side depending on $u$ and its derivatives and linear part given by \eqref{op}.  \\
  \indent
In order to establish our smoothing estimates we need to introduce the weighted Kato-Sobolev
spaces $H^{s_1,s_2}(\R^n),$ $s_1,s_2 \in \R$ defined as 
\begin{equation}\label{wSobolev}
	H^{s_1,s_2}(\R^n)= \{u \in \mathscr{S}'(\R^n): \langle x \rangle^{s_2}\langle D \rangle^{s_1}u \in L^2(\R^n)\},
\end{equation}
where $\langle x \rangle^{s_2}\langle D \rangle^{s_1}$ denotes the operator with symbol $\langle x \rangle^{s_2}\langle \xi \rangle^{s_1}$, $\langle a\rangle$ being the usual Japanese bracket  $(1+|a|^2)^{1/2}$ for $a\in\R^n$. The  main properties of $H^{s_1,s_2}$ spaces are reported in Section \ref{preliminaries}.\\

The main result of this paper reads as follows.

\begin{Th}\label{smoothing_theorem}
	Consider the Cauchy problem \eqref{genCP}, \eqref{op} with $a_j\in C([0,T];{\mathcal B}^\infty),$ $0\leq j\leq p$ and $a_p(t,x)\in\R$. Assume that for every $(t,x)\in[0,T]\times\R$ we have:
\beqs \label{1}
a_p(t,x)&\geq& C_{a_p}>0
\\
\label{2}
|D_x^\beta a_p(t,x)|&\leq&
\frac{C}{\langle x\rangle^{\frac{p-[\beta/2]}{p-1}}},\quad 0\leq\left[\frac\beta 2\right]\leq p-1,\ \beta\ \mbox{odd}\\
	\label{4}
|\Im a_{p-1}(t,x)|&\leq&\frac{C_{p-1}}{\langle 
		x\rangle^{\sigma}},\quad \sigma>1\\
|\Im a_j(t,x)|&\leq&\frac{C_j}{\langle 
		x\rangle^{\frac{j}{p-1}}},\quad 1\leq j\leq p-2\\
	\label{3}
	|\Re D_x^\beta a_j(t,x)|&\leq& C_j \quad 0\leq \beta\leq j-1,\ 3\leq j\leq p-1\\
\label{33}
|\Im D_x^\beta a_j(t,x)|&\leq&\frac{C_j}{\langle 
		x\rangle^{\frac{j-[\beta/2]}{p-1}}},\quad 0<\left[\frac\beta2\right]\leq j-1,\ 3\leq j\leq p-1\\
	\label{a1}
	|\Im D_x a_2|&\leq& 
	\frac{C}{\langle x\rangle^{\frac{1}{p-1}}}
	\eeqs
	for some $C>0$, where $[\beta/2]$ denotes the integer part of $\beta/2$. Then:
	\begin{itemize}
		\item[(i)] Given $m \geq 0$, if $g \in H^m, f \in C([0,T];H^{m})$, then there exists a unique solution to \eqref{genCP}  $u \in C([0,T];H^m)$ satisfying
		{\small
			\begin{align}\label{energy!}
				\|u(t)\|^{2}_{H^{m}} +\ds \int_{0}^{t}\Big(
				\|u(\tau)\|^2_{ H^{ m+\frac{p-1}{2}, -\frac{\sigma}{2}} }+ &\sum_{j = 2}^{p-1}  \|u(\tau)\|^{2}_{ H^{m+ \frac{p-j}{2}, -\frac{p-j}{2(p-1)}} }
				\Big) d\tau \\ \nonumber
				&\leq 
				C_{m,T} \left\{\|g\|^{2}_{H^{m}} + \int_0^t \| f(\tau)\|^{2}_{H^{m}} d\tau   \right\}.
			\end{align}
		}
		
		\item[(ii)]Given $m \geq 0$, if $g \in H^m, f \in L^{2}([0,T];H^{m-\frac{p-1}{2}, \frac{\sigma}{2}} )$, then there exists a unique solution to \eqref{genCP} $u \in C([0,T];H^m)$ satisfying
		{\small
			\begin{align}\label{energy!!}
				\|u(t)\|^{2}_{H^{m}} + \ds\int_{0}^{t}\Big( \|u(\tau)\|^{2}_{ H^{ m+\frac{p-1}{2}, -\frac{\sigma}{2}} } &	+ \sum_{j = 2}^{p-1}  \|u(\tau)\|^{2}_{ H^{  m+ \frac{p-j}{2}, -\frac{p-j}{2(p-1)}} }\Big) d\tau \\ \nonumber
				&\leq 
				C_{m,T} \left\{\|g\|^{2}_{H^{m}} + \int_0^t \|f(\tau)\|^{2}_{H^{m-\frac{p-1}{2},\frac{\sigma}{2}}} d\tau   \right\}.
			\end{align}
		}	
	\end{itemize}
	In both cases,  the positive constant $C_{m,T}$ is of the form $C_me^{C_m'T}$, with $C_m,C'_m>0$, and remains bounded as $T\to 0^+$.
\end{Th}

\begin{Rem}
	We see from the above theorem that the solution to \eqref{genCP} gains $p-1$ derivatives compared to $f$ and $\frac{p-1}{2}$ compared to the datum $g$ (c.f. \eqref{energy!!} and \eqref{energy!}) with respect to some weighted Sobolev norms, where the spatial weight depends on the decay assumption on the subleading term $a_{p-1}(t,x)$ (see \eqref{4}). This is consistent with previous results obtained for Schr\"odinger-type or KdV-type equations, cf. \cite{chihara, KPRV, FS, FT}.
\end{Rem}

\begin{Rem}
The proof of Theorem \ref{smoothing_theorem} is based on a suitable change of variable. This introduces new terms in the operator which allow to trigger an iterative application of sharp 
G{\aa}rding inequality with 
remainders, cf. Theorem \ref{thmg} below. In the application of this strategy, the precise assumptions \eqref{4}-\eqref{a1} on the lower order terms appearing in the theorem play a crucial role. The outcome is the transformation of the operator into a sum of suitable smoothing terms, of positive terms and of purely imaginary ones. In the application of the energy method the positive terms and the purely imaginary terms can be neglected and this leads the desired estimates \eqref{energy!} and \eqref{energy!!}.
\end{Rem}

As a classical application of smoothing estimates, similarly to \cite{FS,FT, KPRV} for the cases $p=2,3,$ in Section \ref{nonlinearapp} we prove a result of local existence and uniqueness for the nonlinear Cauchy problem
\begin{equation}\label{nonlinearCP_intro}
	\begin{cases} Pu = Q(u,\bar u ,D_x u, \ldots, D_x^{p-1}u)\\ u(0,x)=g(x)\end{cases}, \qquad (t,x) \in [0,T]\times \R,
\end{equation}
where $Q(u, \bar u, D_x u, \ldots, D_x^{p-1}u)$ is a polynomial in $u, \bar u , D_x u, \ldots, D_x^{p-1}u$ without constant or linear terms and $g \in \mathscr{S}(\R)$. This result is the content of Theorem \ref{mainNL}, which states, roughly speaking, that under the assumptions of Theorem \ref{smoothing_theorem} on $P$ the Cauchy problem \eqref{nonlinearCP_intro} admits a unique local in time solution $u$ such that 

\beqsn \|u(t,\cdot)\|^2_{L^\infty_t H_x^m} &+& \int_0^{T^\ast}\!\!\! \left( \|u(\tau)\|^2_{H^{m+\frac{p-1}2, -\frac\sigma2}}+\sum_{k=2}^{p-1}
	\|u(\tau)\|^2_{H^{m+\frac{p-k}2, -\frac{p-k}{2(p-1)}}}\right)d\tau 
	\\\nonumber
	&+& \| u \|_{L^\infty_t H_x^{\tilde m, 2N}}^2+\| \partial_tu \|_{L^\infty_t H_x^{\tilde m, 2N}}^2 <\infty\eeqsn
for suitable large enough $m$ and $\tilde m (<m)$ and a small enough $T^\ast<T$, where $\sigma$ is the decay rate appearing in \eqref{4} and $2N$ is the smallest even number larger or equal to $\sigma$.
\\

We conclude this introduction by comparing our results with other similar results for variable coefficient operators, in particular with \cite{chihara} and \cite{FT}, highliting improvements and restrictions with respect to these papers. \\
With respect to the smoothing estimates obtained in \cite{chihara}, we observe that in Theorem \ref{smoothing_theorem} we 
do not assume the same decay condition on the whole symbol of the operator, but we give less and less restictive decay conditions as the order of the corresponding terms decreases; moreover, under these weaker assumptions, we 
obtain also intermediate estimates of the gain of derivatives and polynomial decay.  More to the point, we impose weaker decay conditions on the lower order coefficients $a_j, j=0,\ldots,p-1,$ with respect to \cite{chihara} emphasizing their specific roles in the estimate like in \cite{ABZsuff}. Moreover, our paper contains an application of smoothing estimates to nonlinear Cauchy problems that, as far as we know, has never apperared in literature but in the cases $p=2,3$. Notice that in our paper we do not need to assume a non-trapping condition as in \cite{chihara} but under the conditions \eqref{1} and \eqref{2} of Theorem \ref{smoothing_theorem} such a condition is satisfied, cf. Remark \ref{nontrappingremark} below.
\\ 
With respect to \cite{FT} the main improvement consists in the fact that we are able to treat operators of arbitrary order $p$, thanks to a strategy based on an iterative application of sharp G{\aa}rding inequality whereas the authors of \cite{FT} use Fefferman-Phong inequality to treat second order terms and sharp G{\aa}rding inequality to deal with the first order terms limiting their analysis to third order operators. Therefore, we can apply our results to nonlinear dispersive equations of arbitrary order with variable coefficients governing wave propagation phenomena in inhomogeneous media, cf. \cite{Molinet, Gomez, Israwi}. Another remarkable fact is that in \cite{FT} the authors assume that the seminorms of the subprincipal symbol of the operator are sufficiently small (cf. \cite{FT} condition (3.2)); in our paper, this \textit{smallness} condition is not necessary. Finally, in Theorem 4.2, we admit for the imaginary part of the subleading coefficient $\Im a_{p-1}$ an arbitrary decay rate $\sigma>1$, while in \cite{FT} and in the previous literature the decay was forced to be of the form $\sigma =2N, N \in \N.$ This improvement is relevant because, as we can see from the estimates \eqref{energy!} and \eqref{energy!!}, the loss in the second Sobolev index appearing in $u$ is larger if $\sigma$ is large. Analogously, in \eqref{energy!!} the assumption $f \in L^2([0,T], H^{(p-1)/2, \sigma/2)} (\R)$ is weaker if $\sigma $ is close to $1$. So in our analysis it is convenient to consider $\sigma$ as close as possible to $1$: there is no gain from a stronger decay of $\textrm{Im}\, a_{p-1}$. 
\\
On the other hand, with respect to both \cite{chihara} and \cite{FT}, our result is limited to one space dimension. The dimensional limit is due to the fact that at the moment a more general version of the change of variable used in the proof of Theorem \ref{smoothing_theorem} is out of reach. We intend to investigate this challenging problem in the future and to extend the present results in arbitrary dimension. Nevertheless, the large number of dispersive equations in one space dimension appearing in the literature sufficiently motivates the results presented here.

	\section{Preliminaries}\label{preliminaries}
	\textbf{Notation:} In the sequel we shall denote by $\N$ the set of positive integers and by $\N_0$ the set $\N \cup \{0\}.$ In the proof of Theorem \ref{smoothing_theorem}, for technical reasons, we shall use the modified Japanese bracket $\langle x \rangle_h := (h^2+|x|^2)^{1/2}, x \in \R^n$, where $h \geq 1$ is a real parameter. \\
	 
In the proofs of our results we shall freely use pseudodifferential operators with symbols in the standard H\"ormander classes $\textbf{\textrm{S}}^m(\R^{2n})$ or in the classes $\textbf{\textrm{SG}}(\R^{2n})$, cf. \cite{Cordes, Schrohe, Parenti}. In this section we briefly recall the main definition and the classical properties of these classes. Moreover, we recall the main properties of the weighted Sobolev spaces defined in the Introduction. 
Although we will use these tools in the case of one space dimension we state them in arbitrary dimension in view of future applications. \\

\begin{Def}
	Given $m \in \R$, we shall denote by $\textbf{\textrm{S}}^m(\R^{2n})$ the space of all functions $p \in C^\infty(\R^{2n})$ such that for all $\alpha, \beta \in \N_0^n$ there exists $C_{\alpha \beta}>0$ for which  $$|\partial_\xi^\alpha \partial_x^\beta p(x,\xi)| \leq C_{\alpha \beta}\langle \xi \rangle^{m-|\alpha|} \qquad \forall (x,\xi) \in \R^n \times \R^n.$$
\end{Def}
We recall that $\textbf{\textrm{S}}^m(\R^{2n})$ is a Fr\'echet space endowed with the seminorms given by $$\sup_{(x,\xi)\in \R^n \times \R^n}\langle \xi\rangle^{-m+|\alpha|}|\partial_\xi^\alpha \partial_x^\beta p(x,\xi)|, \qquad \alpha, \beta \in \N_0^n.$$

\begin{Def}
	Given $m_1,m_2 \in \R$, we shall denote by $\textbf{\textrm{SG}}^{m_1,m_2}(\R^{2n})$ the space of all functions $p \in C^\infty(\R^{2n})$ such that for all $\alpha, \beta \in \N_0^n$ there exists $C_{\alpha \beta}>0$ for which  $$|\partial_\xi^\alpha \partial_x^\beta p(x,\xi)| \leq C_{\alpha \beta}\langle \xi \rangle^{m_1-|\alpha|} \langle x \rangle^{m_2-|\beta|} \qquad \forall (x,\xi) \in \R^n \times \R^n.$$
\end{Def}
As before we can endow $\textbf{\textrm{SG}}^{m_1,m_2}(\R^{2n})$ with a Fr\'echet space topology using the family of seminorms $$\sup_{(x,\xi)\in \R^n \times \R^n}\langle \xi\rangle^{-m_1+|\alpha|}\langle x \rangle^{-m_2+|\beta|}|\partial_\xi^\alpha \partial_x^\beta p(x,\xi)|, \qquad \alpha, \beta \in \N_0^n.$$

Given $p \in \textbf{\textrm{S}}^m(\R^{2n})$ (or $p \in \textbf{\textrm{SG}}^{m_1,m_2}(\R^{2n})$), we will denote by $p(x,D)$ (or also by $\textrm{op}(p(x,\xi)$)  the associated pseudodifferential operator 
$$p(x,D)f(x)= \int_{\R^n} e^{ix\cdot \xi} p(x,\xi) \hat{f}(\xi)\, \dslash \xi, \qquad f \in \mathscr{S}(\R^n),$$ where $\dslash \xi=(2\pi)^{-n}d\xi,$ according with the classical Kohn-Nirenberg quantization. \\
We recall that $p(x,D)$ is a linear continuous maps from $\mathscr{S}(\R^n)$ to $\mathscr{S}(\R^n)$ and extends to a linear and continuous map from  $\mathscr{S}'(\R^n)$ to $\mathscr{S}'(\R^n).$ Moreover, if $p \in \textbf{\textrm{S}}^m(\R^{2n}),$ then $p(x,D)$ extends to a bounded operator from $H^s(\R^n)$ to $H^{s-m}(\R^n).$ 
For operators with symbols in $\textbf{\textrm{SG}}^{m_1,m_2}(\R^{2n})$ more precise results can be obtained in weighted Sobolev spaces. Namely, if $p \in \textbf{\textrm{SG}}^{m_1,m_2}(\R^{2n})$, then for every $s_1,s_2 \in \R$,
the operator $p(x,D)$ maps continuously $H^{s_1,s_2}(\R^n)$ into $H^{s_1-m_1,s_2-m_2}(\R^n).$ Composition properties of pseudodifferential operators with symbols in $\textrm{\textbf{SG}}$ classes are collected in the following proposition.

\begin{Prop}
	Let $p \in \textbf{\textrm{SG}}^{m_1,m_2}(\R^{2n}), q \in \textbf{\textrm{SG}}^{m'_1,m'_2}(\R^{2n})$. Then:
	\begin{itemize}\item[(i)] There exists $s \in \textbf{\textrm{SG}}^{m_1+m'_1,m_2+m'_2}(\R^{2n})$ and $r \in \mathscr{S}(\R^{2n})$ such that $$p(x,D)q(x,D)= s(x,D)+r(x,D)$$
		and for every $N \in \N,$ we have
		$$s(x,\xi) - \sum_{|\alpha| <N} \alpha!^{-1}\partial_\xi^\alpha p(x,\xi) D_x^\alpha q(x,\xi) \in \textbf{\textrm{SG}}^{m_1+m'_1-N,m_2+m'_2-N}(\R^{2n}).$$
		\item[(ii)] There exists $\tilde s \in \textbf{\textrm{SG}}^{m_1+m'_1-1,m_2+m'_2-1}(\R^{2n})$ and $\tilde r \in \mathscr{S}(\R^{2n})$ such that $$[p(x,D), q(x,D)]= p(x,D)q(x,D)-q(x,D)p(x,D)= \tilde s(x,D)+ \tilde r(x,D)$$
		and for every $N \in \N,$ we have
		$$\tilde s(x,\xi) -\hskip-0.2cm \sum_{1\leq |\alpha| <N} \ds\frac1{\alpha!}\left\{\partial_\xi^\alpha p(x,\xi) D_x^\alpha q(x,\xi) -\partial_\xi^\alpha q(x,\xi) D_x^\alpha p(x,\xi)\right\}\in \textbf{\textrm{SG}}^{m_1+m'_1-N,m_2+m'_2-N}(\R^{2n}).$$
	\end{itemize}
\end{Prop}

We now recall some properties of weighted Sobolev spaces $H^{s_1,s_2}(\R^n)$ defined by \eqref{wSobolev}. For proofs and details we refer to \cite[Section 3.3]{Cordes} and \cite[Proposition 2.2]{AscanelliCappielloLL2006}. 

\begin{Prop}
	\label{wSobolevproperties}
	Consider for every $s_1,s_2\in \R,$ the space  $H^{s_1,s_2}(\R^n)$ defined by \eqref{wSobolev}. The following properties hold:
	\begin{enumerate}
		\item $H^{s_1,s_2}(\R^n)$  is a Hilbert space endowed with the inner product
		$$(u,v)_{H^{s_1,s_2}} = \left( \px^{s_2} \pd^{s_1}u, \px^{s_2} \pd^{s_1}v \right)_{L^2}.$$
		\item $H^{s_1,s_2}(\R^n)$ is the space of all $u \in \mathscr{S}'(\R^n)$ such that
		$\pd^{s_1}(\px^{s_2}u) \in L^2(\R^n)$ and the norms
		$\left\| \px^{s_2} \pd^{s_1} u \right\|_{L^2}$ and $\left\| \pd^{s_1}(\px^{s_2}u) \right\|_{L^2}$ are equivalent. \\
		\item If $s_j \leq t_j, j=1,2,$ then $H^{t_1,t_2}(\R^n) \subseteq H^{s_1,s_2}(\R^n).$ Moreover, if $s_j<t_j,j=1,2,$ then
		the embedding $H^{t_1,t_2}(\R^n) \hookrightarrow H^{s_1,s_2}(\R^n) $ is compact. \\
		\item The following identities hold: $$\bigcap_{s_1,s_2 \in \R}H^{s_1,s_2}(\R^n)= \mathscr{S}(\R^n), \qquad \bigcup_{s_1,s_2 \in \R}H^{s_1,s_2}(\R^n) = \mathscr{S}'(\R^n).$$
		\item If $s_1 \in \N_0,$ then $H^{s_1,s_2}(\R^n)$ is the space of all functions $u \in L^2(\R^n)$ such that
		$\px^{s_2} \partial_x^{\alpha}u(x) \in L^2(\R^n)$ for all $\alpha \in \N_0^n$ with $|\alpha|\leq s_1$ and an equivalent norm is given by
		$$\sum_{|\alpha| \leq s_1} \|\px^{s_2} \partial_x^{\alpha}u(x) \|_{L^2}.$$
		\item If $s_1 > \frac{n}{2}+j$ for some $j \in \N_0,$ then we have the compact embedding
		$$H^{s_1,s_2}(\R^n) \hookrightarrow \{u \in C^j(\R^n): \px^{s_2}\partial^{\alpha}_xu(x)
		\in C_{(o)}(\R^n) \quad \forall |\alpha| \leq j \}$$
		where $C_{(o)}(\R^n)$ is the space of all functions $f \in C(\R^n)$ such that
		$f(x) \rightarrow 0$ when $|x| \rightarrow +\infty.$
		\item The Fourier transformation $\mathcal F$ is a linear and continuous bijection from $H^{s_1,s_2}(\R^n)$ to $H^{s_2,s_1}(\R^n)$.
		\item If $s_1 >n/2$ and $s_2 \geq 0$, then the space $H^{s_1,s_2}(\R^n)$ is an algebra with respect to the product, that is there exists a constant $C=C_{s_1,s_2}>0$ such that for every $u,v \in H^{s_1,s_2}(\R^n)$, we have $uv \in H^{s_1,s_2}(\R^n)$ and the following estimate holds:
		$$\| uv \|_{H^{s_1,s_2}} \leq C \|u\|_{H^{s_1,s_2}} \cdot \|v\|_{H^{s_1,s_2}}.$$
	\end{enumerate}
\end{Prop}

We conclude this section recalling the two main lower bound inequalities for pseudodifferential operators that we shall use in the sequel, that is sharp G{\aa}rding inequality with remainders and Fefferman-Phong inequality.
\begin{Th}\cite[Thm 4.2 p. 130]{KG} \label{thmg}
Let $a \in S^m(\R^n)$ such that $\textrm{Re}\, a(x,\xi) \geq 0$. Then there exist pseudodifferential operators $q(x,D)$ and $r(x,D)$ with symbols respectively $q \in S^m(\R^{2n})$ and $r \in S^{m-1}(\R^{2n})$ such that $a(x,D)= q(x,D)+r(x,D)$ and
$$\textrm{Re}\, \langle q(x,D) u,u \rangle_{L^2} \geq -C \|u\|^2_{H^{\frac{m-1}2}}, \qquad u \in \mathscr{S}(\R^n),$$
$$r(x,\xi) \sim \psi_1(\xi) D_x a(x,\xi) + \sum_{\alpha+\beta \geq 2 }\psi_{\alpha, \beta}(\xi)\partial_\xi^\alpha D_x^\beta a(x,\xi),$$
for some real-valued symbols $\psi_1 \in S^{-1}$ and $\psi_{\alpha, \beta} \in S^{\frac{\alpha-\beta}2}.$  
\end{Th}
\begin{Th} \cite{FP}
	Let $a \in S^m(\R^{2n})$ with $a(x,\xi) \geq 0$. Then 
	$$\langle a(x,D) u,u \rangle_{L^2} \geq -c \|u \|^2_{H^{\frac{m-2}2}}, \qquad u \in \mathscr{S}(\R^n),$$
	for some $c>0$.
	\end{Th}

\section{Smoothing estimates for the linear Cauchy problem}

In this section we prove Theorem \ref{smoothing_theorem}. 
\\

\noindent
{\it Proof of Theorem \ref{smoothing_theorem}.}
For $x, \xi \in \R$ and a large parameter $h \geq 1$ we define (cf. \cite[Section 2]{ABZsuff})
\begin{equation}\label{eq_lmabda_p_1}
\lambda_{p-1}(x,\xi):=M_{p-1} \omega\left(\frac\xi h\right) \int_{0}^{x} \langle y\rangle^{-\sigma} \psi\left( \frac{\langle y\rangle}{\langle\xi\rangle_h^{p-1}}\right)dy,
\end{equation}
and for $j = 2, \ldots, p-1$, 
\begin{equation}\label{eq_lambda_p_k}
\lambda_{p-j}(x,\xi):=M_{p-j} \omega\left(\frac\xi h\right)  \langle\xi\rangle_h^{-j+1} \int_{0}^{x} \langle y\rangle^{-\frac{p-j}{p-1}}
\psi\left( \frac{\langle y\rangle}{\langle\xi\rangle_h^{p-1}}\right)dy,
\end{equation}
where $M_{p-j}$, $j = 1, \ldots, p-1$, are positive constants to be settled later on and the functions $\omega\in C^\infty(\R)$, $\psi\in C^\infty_0(\R)$ satisfy:
\beqsn
&&\omega(y)=
\begin{cases}
	0 & |y|\leq 1\\
	-|y|^{p-1}/y^{p-1}& |y|\geq 2
\end{cases}\\
&&0\leq\psi(y)\leq 1\quad\forall y\in\R, 
\quad \psi(y)=
\begin{cases}
	1& |y|\leq\frac 12\\
	0& |y|\geq1.
\end{cases}
\eeqsn
The symbols $\lambda_{p-j}$, $j = 1, \ldots, p-1$, fulfill the  following estimates that can be derived as in \cite{ABZsuff}:
\beqsn
|\partial^{\alpha}_{\xi}\partial^{\beta}_{x} \lambda_{p-1}(x,\xi)| &\leq& C_{\alpha, \beta}M_{p-1} \langle \xi \rangle^{-\alpha}_{h} \langle x \rangle^{1-\sigma-\beta},
\\
|\partial^{\alpha}_{\xi}\partial^{\beta}_{x} \lambda_{p-j}(x,\xi)| &\leq& C_{\alpha, \beta}M_{p-j} \min\{
\langle \xi \rangle^{1-j}_{h} \langle x \rangle^{1-\frac{p-j}{p-1}},1
\}\langle \xi \rangle^{-\alpha}_{h} \langle x \rangle^{-\beta}, \quad j=2,...,p-1.
\eeqsn
These estimates imply in particular that $\lambda_{p-j} \in \textbf{\textrm{SG}}^{0,0}(\R^2)$ for any $j=1,\ldots,p-1$ because we are assuming $\sigma > 1$.
 Therefore the symbol
$$\Lambda =: \lambda_1 + \cdots + \lambda_{p-1} \in \textbf{\textrm{SG}}^{0,0}(\R^2).$$ Moreover, we also have $e^{\pm \Lambda} \in \textbf{\textrm{SG}}^{0,0}(\R^2)$. 

If the family parameter $h$ is large enough, say $h > h_0(M_{p-1},\ldots,M_1)$, then using Neumann series we see \cite{KB,KW} that the operator $e^{\Lambda}(x,D)=\textrm{op}(e^{\Lambda(x,\xi)})$ is invertible and its inverse is given by
$$
	(e^\Lambda)^{-1} = \hskip2pt ^R (e^{-\Lambda}) \circ \sum_{j\geq0}(-r)^j,
	$$
where $^R (e^{-\Lambda})(x,D)$ is the so-called reverse operator of $e^{- \Lambda}(x,D)$, introduced in \cite[Proposition 2.13]{KW} as the transposed of $e^{- \Lambda}(x,-D)$, defined as an oscillatory integral by
$$^{R}(e^{- \Lambda}(x,D))u(x) = Os - \iint e^{i\xi (x-y) - \Lambda(y,\xi)}u(y)\, dy \dslash \xi,
$$ and
$$
r(x,\xi) \sim \sum_{\alpha \geq 1} \frac{1}{\alpha!} \partial^{\alpha}_{\xi}\{ e^{\Lambda} D^{\alpha}_{x} e^{-\Lambda} \}(x,\xi)
$$
is in $\textbf{\textrm{SG}}^{-1,-1}(\R^2)\subset S^{-1}(\R^2)$ since its leading term is given by $-\partial_\xi D_x\Lambda$. Moreover, given $p\in S^m(\R^{2})$ we have the asymptotic expansion
$${e^\Lambda}(x,D)p(x,D)^R (e^{-\Lambda})(x,D)=p(x,D)+\textrm{op}\left(\sum_{\alpha+\beta \geq 1} \frac{1}{\alpha!\beta!} \partial^{\alpha}_{\xi}\{ \partial_\xi^\beta e^{\Lambda} D_x^\beta p D^{\alpha}_{x} e^{-\Lambda} \}(x,\xi)\right).$$ 

Therefore, we can pass from the original Cauchy problem \eqref{genCP} to the equivalent auxiliary Cauchy problem
\beqs
\label{CP_auxiliary}
\begin{cases}
	e^{\Lambda} P(t,x,D_t,D_x) \{e^{\Lambda}\}^{-1} v(t,x)= e^{\Lambda} f(t,x) & (t,x) \in [0,T]\times\R\cr
	u(0,x)= e^{\Lambda} g(x) & x\in\R
\end{cases}
\eeqs
in the unknown $v=e^\Lambda u$. If we find a solution $v$ to \eqref{CP_auxiliary} then $u = \{e^{\Lambda}\}^{-1}v$ solves \eqref{genCP}. In the sequel, choosing the constants $M_{p-j}$ in \eqref{eq_lmabda_p_1} and \eqref{eq_lambda_p_k} large enough, we shall derive a priori smoothing energy estimates for the problem \eqref{CP_auxiliary}.

\noindent 
Working as in \cite{ABZsuff} and \cite[Section 4]{ABmoser} and using the asymptotic expansion here above it is possible to prove that the operator $e^{\Lambda} (iP) \{e^{\Lambda}\}^{-1}$ presents the following structure:
\begin{align*}
		e^{\Lambda} (iP) \{e^{\Lambda}\}^{-1} &= \partial_t + ia_{p}(t,x)D^{p}_{x} \\
		&- \textrm{op}\left( pa_p(t,x)\xi^{p-1} \partial_x \lambda_{p-1}\right) - \Im a_{p-1} D_x^{p-1} + A_{p-1, 2} \\
		&- \textrm{op}\left( pa_p(t,x) \xi^{p-1} \partial_x \lambda_{p-2}\right) +A_{p-2,1} + A_{p-2, 2} \\
		&- \cdots - \\
		&- \textrm{op}\left( pa_p(t,x) \xi^{p-1} \partial_x \lambda_{2}\right) +A_{2,1} + A_{2, 2} \\
		&- \textrm{op}\left( pa_p(t,x) \xi^{p-1} \partial_x \lambda_{1} \right)+A_{1,1} + A_{1, 2} \\
		& + A_0
\end{align*}
where, for $j = 1, \ldots, p-1$ and $k = 1, 2$:

\begin{itemize}
\item[-] $A_{p-j,k}$ is of order $p-j$ and depends only on $M_{p-1}, \ldots, M_{p-(j-1)}$, 
\item[-] $A_{p-j, 1}$ is real-valued,
$$
|A_{p-j,1}(t,x,\xi)| \leq C(M_{p-1},\ldots, M_{p-(j-1)}) \langle \xi \rangle^{p-j}_{h} \langle x \rangle^{-\frac{p-j}{p-1}},
$$
\beqs\label{i}
\Re\, \langle A_{p-j,2}(t,\cdot,D) u(t,\cdot), u(t,\cdot) \rangle_{L^2} = 0,
\eeqs
\item[-] $A_0$ is of order zero.
\end{itemize} 
Let us remark that, by \eqref{i}, we may not worry about the terms $A_{p-j,2}$ in the energy method.

Now, from the definition of the symbols $\lambda_{p-j}$, if we localize the operator $e^{\Lambda} (iP) \{e^{\Lambda}\}^{-1}$ in the zone $|\xi| \geq 2h$ we obtain
\begin{align*}
	e^{\Lambda} iP \{e^{\Lambda}\}^{-1} &= \partial_t + ia_{p}(t,x)D^{p}_{x} \\
	&+ M_{p-1}\textrm{op}\left( pa_p(t,x)  |\xi|^{p-1} \langle x \rangle^{-\sigma}\right) - \Im a_{p-1} D_x^{p-1} + A_{p-1, 2} \\
	&+ M_{p-2}\textrm{op}\left( pa_p(t,x)|\xi|^{p-1} \langle \xi \rangle^{-1}_{h} \langle x \rangle^{-\frac{p-2}{p-1}} \right)+A_{p-2,1} + A_{p-2, 2} \\
	&+ \cdots + \\
	&+ M_{2}\textrm{op}\left( pa_p(t,x) |\xi|^{p-1} \langle \xi \rangle_{h}^{-(p-3)} \langle x \rangle^{-\frac{2}{p-1}}\right) +A_{2,1} + A_{2, 2} \\
	&+ M_{1}\textrm{op}\left( p a_p(t,x) |\xi|^{p-1} \langle \xi \rangle^{-(p-2)}_{h} \langle x \rangle^{-\frac{1}{p-1}}\right) +A_{1,1} + A_{1, 2} \\
	& + A_0.
\end{align*}

Next we shall apply sharp G{\aa}rding inequality with remainders to the terms of order $p-1, p-2,...,3$ in $e^{\Lambda} iP \{e^{\Lambda}\}^{-1}$, the Fefferman-Phong inequality to the terms of order 2 and finally the sharp G{\aa}rding inequality to the terms of order 1, in order to produce smoothing energy estimates. This is the reason why the assumption \eqref{a1} on $\textrm{Im}\, D_x a_2$ is weaker than the one for higher order coefficients, that is \eqref{33}; indeed, Fefferman-Phong and sharp G{\aa}rding inequality do not produce real-valued remainders. We shall explain the procedure in detail for the terms of order $p-1$. We have (in the zone $|\xi| \geq 2h$, where $|\xi|\geq \frac2{\sqrt 5}\langle\xi\rangle_h$)
\beqsn
	 M_{p-1}pa_p(t,x) |\xi|^{p-1} \langle x \rangle^{-\sigma} - \Im a_{p-1} (t,x)\xi^{p-1} &\geq& \left\{C_{a_p}\left(\frac2{\sqrt 5}\right)^{p-1} M_{p-1} - C_{p-1}\right\} \langle x \rangle^{-\sigma}\langle \xi \rangle^{p-1}_{h}
\\
&:=&C_{p-1}  \langle x \rangle^{-\sigma}\langle \xi \rangle^{p-1}_{h}.
\eeqsn
So, if $M_{p-1}$ is large enough we get $C_{p-1} > 0$ and therefore
$$
c_{p-1}(t,x,\xi) := M_{p-1}pa_p(t,x) |\xi|^{p-1} \langle x \rangle^{-\sigma} - \Im a_{p-1} (t,x)\xi^{p-1} - C_{p-1}\langle x \rangle^{-\sigma} \langle \xi \rangle^{p-1}_{h} \geq 0.
$$
Sharp G{\aa}rding inequality applied to $c_{p-1}(t,x,D) $ then gives (see Theorem $4.2$ in \cite{KG}):
$$
M_{p-1}\textrm{op}\left( pa_p(t,x)  |\xi|^{p-1} \langle x \rangle^{-\sigma}\right)  - \Im a_{p-1}(t,x) D_x^{p-1} = C_{p-1}\langle x \rangle^{-\sigma} \langle D_x \rangle^{p-1}_{h} + Q_{p-1} + R_{p-2},
$$
where $Q_{p-1}$ is a positive operator in the sense that $$\Re\, \langle Q_{p-1}(t,x,D)u, u \rangle \geq 0, \qquad\forall u \in \mathscr{S}(\R),$$ and 
$$
R_{p-2} (t,x,\xi)\sim 
\psi_1(\xi)D_x  c_{p-1}(t,x,\xi) +\sum_{\alpha+\beta\geq2}\psi_{\alpha,\beta}(\xi)\partial_\xi^\alpha D_x^\beta c_{p-1}(t,x,\xi), 
$$
for suitable symbols $\psi_{\alpha, \beta}(\xi) \in S^{\frac{\alpha - \beta}{2}}$, $\psi_{\beta}(\xi) \in S^{-1}$. 
Due to the hypothesis on $a_{p-1}$ (following \cite{ABZsuff} once more) it turns out that the remainder $R_{p-2}$ has the following structure 

\begin{align*}
	R_{p-2} = A'_{p-2,1} + A'_{p-2, 2} 
	+ \cdots + 
	 A'_{2,1} + A'_{2, 2} 
	+ A'_{1,1} + A'_{1, 2}
	+ A'_0, 
\end{align*}
for some operators $A'_{p-j,k}$ satisfying the same properties as $A_{p-j,k}$. In this way we get, for some new operators $A_{p-j,k}$ satisfying the same properties as before:
\begin{align*}
	e^{\Lambda} iP \{e^{\Lambda}\}^{-1} &= \partial_t + ia_{p}(t,x)D^{p}_{x} \\
	&+ C_{p-1}\langle x \rangle^{-\sigma} \langle D_x \rangle^{p-1}_{h} + Q_{p-1} + A_{p-1, 2} \\
	&+ M_{p-2}\textrm{op}\left( pa_p(t,x) |\xi|^{p-1} \langle \xi \rangle^{-1}_{h} \langle x \rangle^{-\frac{p-2}{p-1}}\right) +A_{p-2,1} + A_{p-2, 2} \\
	&+ \cdots + \\
	&+ M_{2}\textrm{op}\left( pa_p(t,x) |\xi|^{p-1} \langle \xi \rangle_{h}^{-(p-3)} \langle x \rangle^{-\frac{2}{p-1}}\right) +A_{2,1} + A_{2, 2} \\
	&+ M_{1}\textrm{op}\left( p a_p(t,x) |\xi|^{p-1}\langle \xi \rangle^{-(p-2)}_{h} \langle x \rangle^{-\frac{1}{p-1}} \right) +A_{1,1} + A_{1, 2} \\
	& + A_0.
\end{align*}
Next we observe that 
$$
C_{p-1}\langle x \rangle^{-\sigma} \langle D \rangle^{p-1}_{h} = C_{p-1} \{\langle D \rangle^{\frac{p-1}{2}}_{h} \langle x \rangle^{-\frac{\sigma}{2}} \} \{\langle x \rangle^{-\frac{\sigma}{2}} \langle D \rangle^{\frac{p-1}{2}}_{h}\}  - C_{p-1} [\langle D \rangle^{\frac{p-1}{2}}_{h},\langle x \rangle^{-\sigma} ] \langle D \rangle^{\frac{p-1}{2}}_{h}.
$$
Since the asymptotic expansion of the symbol of $C_{p-1}[\langle D \rangle^{\frac{p-1}{2}}_{h},\langle x \rangle^{-\sigma} ] \langle D \rangle^{\frac{p-1}{2}}_{h}$ is given by
$$
 C_{p-1}\sum_{\alpha\geq 1} \frac{1}{\alpha!} \partial^{\alpha}_{\xi} \langle \xi \rangle^{\frac{p-1}{2}}_{h} D^{\alpha}_{x} \langle x \rangle^{-\sigma} \langle \xi \rangle^{\frac{p-1}{2}}_{h}  \in \textbf{\textrm{SG}}^{p-2, -\sigma-1}(\R^2)
$$
and $C_{p-1}$ depends only on $M_{p-1}$, we can write
\begin{align*}
	C_{p-1} [\langle D \rangle^{\frac{p-1}{2}}_{h},\langle x \rangle^{-\sigma} ] \langle D \rangle^{\frac{p-1}{2}}_{h} = A"_{p-2,1} + A"_{p-2, 2} 
	+ \cdots + 
	 A"_{2,1} + A"_{2, 2} + A"_{1,1} + A"_{1, 2} 
	+ A"_0, 
\end{align*}
for some operators $A"_{p-j,k}$ satisfying the same properties as $A_{p-j,k}$. In a nutshell, in the zone $|\xi| \geq 2h$ we can rewrite, for some new operators $A_{p-j,k}$ satisfying the same properties as before, the conjugated operator as

\begin{align*}
	e^{\Lambda} (iP) \{e^{\Lambda}\}^{-1} &= \partial_t + ia_{p}(t,x)D^{p}_{x} \\
	&+ C_{p-1} \{\langle D \rangle^{\frac{p-1}{2}}_{h} \langle x \rangle^{-\frac{\sigma}{2}} \} \{\langle x \rangle^{-\frac{\sigma}{2}} \langle D \rangle^{\frac{p-1}{2}}_{h}\} + Q_{p-1} + A_{p-1, 2} \\
	&+ M_{p-2}\textrm{op}\left( pa_p(t,x) |\xi|^{p-1} \langle \xi \rangle^{-1}_{h} \langle x \rangle^{-\frac{p-2}{p-1}}\right) +A_{p-2,1} + A_{p-2, 2} \\
	&+ \cdots + \\
	&+ M_{2}\textrm{op}\left( pa_p(t,x) |\xi|^{p-1} \langle \xi \rangle_{h}^{-(p-3)} \langle x \rangle^{-\frac{2}{p-1}}\right) +A_{2,1} + A_{2, 2} \\
	&+ M_{1}\textrm{op}\left( p a_p(t,x) |\xi|^{p-1} \langle \xi \rangle^{-(p-2)}_{h} \langle x \rangle^{-\frac{1}{p-1}}\right) +A_{1,1} + A_{1, 2} \\
	& + A_0
\end{align*}
with $C_{p-1}=:C_{M_{p-1}}$ a positive constant depending only on $M_{p-1}$.
Repeating this process again $p-4$ times, and then applying Fefferman-Phong inequality to the terms of order 2 and sharp G{\aa}rding inequality to terms of level 1, following \cite{ABZsuff} again, we end up with 
\begin{align*}
	e^{\Lambda} (iP) \{e^{\Lambda}\}^{-1} &= \partial_t + ia_{p}(t,x)D^{p}_{x} \\
	&+ C_{M_{p-1}} \{\langle D \rangle^{\frac{p-1}{2}}_{h} \langle x \rangle^{-\frac{\sigma}{2}} \} \{\langle x \rangle^{-\frac{\sigma}{2}} \langle D \rangle^{\frac{p-1}{2}}_{h}\} + Q_{p-1} + A_{p-1, 2} \\
	&+ C_{M_{p-1}, M_{p-2}} \{\langle D \rangle^{\frac{p-2}{2}}_{h} \langle x \rangle^{-\frac{p-2}{2(p-1)}} \} \{\langle x \rangle^{-\frac{p-2}{2(p-1)}} \langle D \rangle^{\frac{p-2}{2}}_{h}\} + Q_{p-2}+ A_{p-2, 2} \\
	&+ \cdots + \\
	&+ C_{M_{p-1}, M_{p-2}, \ldots, M_{2}} \{\langle D \rangle^{\frac{2}{2}}_{h} \langle x \rangle^{-\frac{2}{2(p-1)}} \} \{\langle x \rangle^{-\frac{2}{2(p-1)}} \langle D \rangle^{\frac{2}{2}}_{h}\} + Q_{2}+ A_{2, 2} \\
	&+ C_{M_{p-1}, M_{p-2}, \ldots, M_{2}, M_{1}}\{\langle D \rangle^{\frac{1}{2}}_{h} \langle x \rangle^{-\frac{1}{2(p-1)}} \} \{\langle x \rangle^{-\frac{1}{2(p-1)}} \langle D \rangle^{\frac{1}{2}}_{h}\} + Q_{1} + A_{1, 2} \\
	& + A_0,
\end{align*}
where, for $j = 1, \ldots, p-1$, the positive constants $C_{M_{p-1}, M_{p-2}, \ldots, M_{p-j}}$ depend only on $M_{p-1},...,M_{p-j}$, the operators $Q_{p-j}$ are positive in the sense that
$$\Re\, \langle Q_{p-j}(t,x,D)u, u \rangle \geq 0, \qquad\forall u \in \mathscr{S}(\R),$$ and $$\Re\,\langle A_{p-j,2}u,u \rangle = 0.$$

Finally, writing 
\begin{align*}
\partial_t &= e^{\Lambda} (iP) \{e^{\Lambda}\}^{-1} - ia_p(t,x)D^{p}_{x} \\
&- C_{M_{p-1}} \{\langle D \rangle^{\frac{p-1}{2}}_{h} \langle x \rangle^{-\frac{\sigma}{2}} \} \{\langle x \rangle^{-\frac{\sigma}{2}} \langle D \rangle^{\frac{p-1}{2}}_{h}\} + Q_{p-1} + A_{p-1, 2} \\
	&- C_{M_{p-1}, M_{p-2}} \{\langle D \rangle^{\frac{p-2}{2}}_{h} \langle x \rangle^{-\frac{p-2}{2(p-1)}} \} \{\langle x \rangle^{-\frac{p-2}{2(p-1)}} \langle D \rangle^{\frac{p-2}{2}}_{h}\} + Q_{p-2}+ A_{p-2, 2} \\
	&- \cdots - \\
	&- C_{M_{p-1}, M_{p-2}, \ldots, M_{2}} \{\langle D \rangle^{\frac{2}{2}}_{h} \langle x \rangle^{-\frac{2}{2(p-1)}} \} \{\langle x \rangle^{-\frac{2}{2(p-1)}} \langle D \rangle^{\frac{2}{2}}_{h}\} + Q_{2}+ A_{2, 2} \\
	&- C_{M_{p-1}, M_{p-2}, \ldots, M_{2}, M_{1}}\{\langle D \rangle^{\frac{1}{2}}_{h} \langle x \rangle^{-\frac{1}{2(p-1)}} \} \{\langle x \rangle^{-\frac{1}{2(p-1)}} \langle D \rangle^{\frac{1}{2}}_{h}\} + Q_{1} + A_{1, 2} \\
	& - A_0,
\end{align*}
we obtain (using the fact that $a_p(t,x)$ is real-valued)

\begin{align*}
		\partial_{t} \|v(t)\|^{2}_{L^2} &= 2 \Re\, \langle \partial_t v(t), v(t) \rangle \\
		&= 2 \Re\, \langle e^{\Lambda} (iP) \{e^{\Lambda}\}^{-1} v(t), v(t) \rangle - 2 \sum_{j = 1}^{p-1} \Re\, \langle Q_{p-j} v(t), v(t) \rangle\\
		&- C_{M_{p-1}} \|v(t)\|^{2}_{ H^{\frac{p-1}{2}, -\frac{\sigma}{2}} } - \sum_{j = 2}^{p-1} C_{M_{p-1}, \ldots, M_{p-(j+1)}} \|v(t)\|^{2}_{ H^{ \frac{p-j}{2}, -\frac{p-j}{2(p-1)}} },
\end{align*}
which implies
\begin{align}\label{A}
\nonumber
\partial_{t} \|v(t)\|^{2}_{L^2} &\leq \|e^{\Lambda} P \{e^{\Lambda}\}^{-1} v(t)\|_{L^2}^{2} + \|v(t)\|_{L^2}^{2}  \\
&- C_{M_{p-1}} \|v(t)\|^{2}_{ H^{ \frac{p-1}{2}, -\frac{\sigma}{2} } } - \sum_{j = 2}^{p-1} C_{M_{p-1}, \ldots, M_{p-(j+1)}} \|v(t)\|^{2}_{ H^{ \frac{p-j}{2}, -\frac{p-j}{2(p-1)}} }.
\end{align}
On the other hand, noticing that 
\begin{align*}
2 \Re\, \langle e^{\Lambda} (iP) \{e^{\Lambda}\}^{-1} v(t), v(t) \rangle &
= 2 \Re\, \langle \langle D \rangle^{\frac{p-1}{2}}_{h}\langle x \rangle^{-\frac{\sigma}{2}}\langle x \rangle^{\frac{\sigma}{2}} \langle D \rangle^{-\frac{p-1}{2}}_{h} e^{\Lambda} (iP) \{e^{\Lambda}\}^{-1} v(t),   v(t) \rangle \\
&= 2 \Re\, \langle \langle x \rangle^{\frac{\sigma}{2}} \langle D \rangle^{-\frac{p-1}{2}}_{h} e^{\Lambda} (iP) \{e^{\Lambda}\}^{-1} v(t), \langle x \rangle^{-\frac{\sigma}{2}} \langle D \rangle^{\frac{p-1}{2}}_{h} v(t) \rangle \\
&\leq \|e^{\Lambda} (iP) \{e^{\Lambda}\}^{-1}v(t)\|^{2}_{H^{-\frac{p-1}{2},\frac{\sigma}{2}}} + \|v(t)\|^{2}_{H^{\frac{p-1}{2}, -\frac{\sigma}{2}} },
\end{align*}
we also obtain 
\begin{align}\label{B}\nonumber
	\partial_{t} \|v(t)\|^{2}_{L^2} &\leq \|e^{\Lambda} (iP) \{e^{\Lambda}\}^{-1}v(t)\|^{2}_{H^{-\frac{p-1}{2},\frac{\sigma}{2}}} + \|v(t)\|_{L^2}^{2}  \\
	&+(1- C_{M_{p-1}}) \|v(t)\|^{2}_{ H^{\frac{p-1}{2}, -\frac{\sigma}{2}} } - \sum_{j = 2}^{p-1} C_{M_{p-1}, \ldots, M_{p-(j+1)}} \|v(t)\|^{2}_{ H^{\frac{p-j}{2}, -\frac{p-j}{2(p-1)}} }.
\end{align}

Gronwall's inequality then gives, looking respectively to \eqref{A} and \eqref{B},
{\small
\begin{multline*}
\|v(t)\|_{L^2}^{2} + \int_{0}^{t} \Big(\|v(\tau)\|^{2}_{ H^{ \frac{p-1}{2}, -\frac{\sigma}{2}} } 	+ \sum_{j = 2}^{p-1}  \|v(\tau)\|^{2}_{ H^{ \frac{p-j}{2}, -\frac{p-j}{2(p-1)}} }\Big) d\tau \\
\leq 
e^{Ct} \left\{\|v(0)\|_{L^2}^{2} + \int_0^t \|e^{\Lambda} P \{e^{\Lambda}\}^{-1} v(\tau)\|_{L^2}^{2} d\tau   \right\}
\end{multline*}
}
and 
{\small
	\begin{align}
		\|v(t)\|_{L^2}^{2} + \int_{0}^{t} \Big( \|v(\tau)\|^{2}_{ H^{ \frac{p-1}{2}, -\frac{\sigma}{2} } }	+ &\sum_{j = 2}^{p-1}  \|v(\tau)\|^{2}_{ H^{ \frac{p-j}{2}, -\frac{p-j}{2(p-1)}} }\Big) d\tau  \\ \nonumber
		& \leq 
		e^{Ct} \left\{\|v(0)\|_{L^2}^{2} + \int_0^t \|e^{\Lambda} P \{e^{\Lambda}\}^{-1} v(\tau)\|^{2}_{H^{-\frac{p-1}{2},\frac{\sigma}{2}}} d\tau   \right\}.
	\end{align}
}
To obtain the above energy inequalities for $H^{m}$ Sobolev norms, we just apply the same procedure to the operator $\langle D_x \rangle^{m} P \langle D_x \rangle^{-m}$ which has the same structure as $P$, see \cite{ABZsuff} for more details. Therefore one gets, respectively,
{\small
	\begin{align*}
		\|v(t)\|^{2}_{H^{m}} + \int_{0}^{t} \Big( \|v(\tau)\|^{2}_{ H^{m+\frac{p-1}{2}, -\frac{\sigma}{2} } }	+ &\sum_{j = 2}^{p-1}  \|v(\tau)\|^{2}_{ H^{ m+ \frac{p-j}{2}, -\frac{p-j}{2(p-1)}} } \Big) d\tau \\ \nonumber
		&\leq 
		C_{m, T} \left\{\|v(0)\|^{2}_{H^{m}} + \int_0^t \|e^{\Lambda} P \{e^{\Lambda}\}^{-1} v(\tau)\|^{2}_{H^{m}} d\tau  \right\}
	\end{align*}
}
and
{\small
	\begin{align*}
		\|v(t)\|^{2}_{H^{m}} + \int_{0}^{t} \Big( \|v(\tau)\|^{2}_{ H^{ \left(m+\frac{p-1}{2}, -\frac{\sigma}{2} \right)} }	+ &\sum_{j = 2}^{p-1}  \|v(\tau)\|^{2}_{ H^{ \left(m+ \frac{p-j}{2}, -\frac{p-j}{2(p-1)} \right)} } \Big) d\tau \\
		& \leq 
		C_{m,T} \left\{\|v(0)\|^{2}_{H^{m}} + \int_0^t \|e^{\Lambda} P \{e^{\Lambda}\}^{-1} v(\tau)\|^{2}_{H^{m-\frac{p-1}{2},\frac{\sigma}{2}}} d\tau   \right\}.
	\end{align*}
}

From the above a priori smoothing energy estimates one gets the following auxiliary result:

\begin{Prop}
	Consider the operator $P$ in \eqref{op} under the hypotheses of Theorem \ref{smoothing_theorem}. Then 
	\begin{itemize}
		\item[(i)] Given $m \geq 0 $, $g \in H^m, f \in C([0,T];H^{m})$, there exists a unique solution $v \in C([0,T];H^m)$ to \eqref{CP_auxiliary} satisfying
		\begin{align}\label{energy!!!}
			\|v(t)\|^{2}_{H^{m}} + \int_{0}^{t} \Big( \|v(\tau)\|^{2}_{ H^{m+\frac{p-1}{2}, -\frac{\sigma}{2} } }	+ &\sum_{j = 2}^{p-1}  \|v(\tau)\|^{2}_{ H^{ m+ \frac{p-j}{2}, -\frac{p-j}{2(p-1)}} } \Big) d\tau \\ \nonumber
			&\leq 
			C_{m, T} \left\{\|e^{\Lambda} g\|^{2}_{H^{m}} + \int_0^t \|e^{\Lambda} f(\tau)\|^{2}_{H^{m}} d\tau  \right\}
		\end{align}
		\item[(ii)] Given $m \geq 0 $, $g \in H^m, f \in L^{2}([0,T];H^{m-\frac{p-1}{2}, \frac{\sigma}{2}})$, there exists a unique solution $v \in C([0,T];H^m)$ to \eqref{CP_auxiliary} satisfying
		{\small
			\begin{align}\label{energy!!!!}
				\|v(t)\|^{2}_{H^{m}} + \int_{0}^{t} \Big( \|v(\tau)\|^{2}_{ H^{ m+\frac{p-1}{2}, -\frac{\sigma}{2}} } + &\sum_{j = 2}^{p-1}  \|v(\tau)\|^{2}_{ H^{ m+ \frac{p-j}{2}, -\frac{p-j}{2(p-1)}} } \Big) d\tau \\ \nonumber
				& \leq 
				C_{m,T} \left\{\|e^{\Lambda} g\|^{2}_{H^{m}} + \int_0^t \|e^{\Lambda}f(\tau)\|^{2}_{H^{m-\frac{p-1}{2},\frac{\sigma}{2}}} d\tau   \right\}.
			\end{align}
		}
	\end{itemize}
\end{Prop}

Using that $\{e^{\Lambda}\}^{\pm1}$ have order zero and the inequalities \eqref{energy!!!}, \eqref{energy!!!!}, we see that the solution $u = \{e^{\Lambda}\}^{-1} v$ of \eqref{genCP} satisfies the smoothing energy inequalities \eqref{energy!} and \eqref{energy!!}.

\qed

\begin{Rem}\label{nontrappingremark}
	Notice that to prove Theorem \ref{smoothing_theorem} we do not need to impose a non-trapping condition for the principal symbol of our operator as in \cite{chihara, KPRV}. Nevertheless, we observe that the assumptions \eqref{1} and \eqref{2} imply such a condition. As a matter of fact, if we consider the Hamilton system 
	\begin{equation}
	\begin{cases}
		\dot x(s)=p a_p(t,x(s))\xi(s)^{p-1},\\[2mm]
		\dot\xi(s)=-\partial_x a_p(t,x(s))\xi(s)^p
	\end{cases}
\end{equation}
with initial condition $(x(0), \xi(0))=(x_0,\xi_0)$ for $\xi_0 \neq 0$, then
using the change of variable
  \begin{equation}\label{eq:rho-def}
  	\rho(s):=\frac{1}{\xi(s)^{p-1}},
  \end{equation}
we are reduced to the equivalent system
\begin{equation}\label{eq:rho-system}
	\begin{cases}
		\dot x(s)=\displaystyle \frac{p a_p(t,x(s))}{\rho(s)},\\[3mm]
		\dot\rho(s)=(p-1)\partial_x a_p(t,x(s)).
	\end{cases}
\end{equation}
From the second equation we have that 
$$ |{\dot\rho(s)}|\leq (p-1)\sup_{(t,x)\in [0,T] \times \R}|{\partial_x a_p(t,x)}|=:M <+\infty.
$$
Integrating between $0$ and $s>0$ we get	\begin{equation}\label{eq:rho-linear-bound}
	|\rho(s)|
	\leq |\rho(0)|+M s.
\end{equation}
Since $\xi(s) \neq 0$ for any $s \in \R$, then by \eqref{eq:rho-def}, $\rho(s)$ has constant sign and then the same holds for $\dot x(s)$. Then \begin{equation}\label{eq:x-speed-lower}
		|\dot x(s)|
		=\frac{p a_p(t,x(s))}{|\rho(s)|}
		\geq
		\frac{p C_p}{|\rho(0)|+M s}.
	\end{equation}
	Integrating between $0$ and $s>0$ we get
	\[
	|x(s)-x(0)|
	\geq
	\int_0^s 
	\displaystyle\frac{p C_p}{|\rho(0)|+M s'} \, ds'=	\displaystyle\frac{p C_p}{M}
	\log\left(1+\frac{M s}{|\rho(0)|}\right).\]

Since the right-hand side tends to $+\infty$ then also  $|x(s)|\to +\infty$ for $s \to +\infty$, hence $x(s)$ escapes from any compact subset of $\R$ for $s \to +\infty.$
\end{Rem}

\section{Application to the nonlinear Cauchy problem}\label{nonlinearapp}

Our next goal is to deal with the following problem 
\begin{equation} \label{nonlinearCP} \begin{cases}
		P(t,x,D_t,D_x)u= Q(u,\bar u, D_x^r u)  \\ u(0,x)=g(x) \end{cases}, \qquad (t,x) \in [0,T]\times \R
\end{equation}
where $P$ is of the form \eqref{op} and satisfies the assumptions of Theorem \ref{smoothing_theorem} and $Q$ is a nonlinear term of the form 
\begin{equation}\label{nonlinearterm}
	Q(u,\bar u, D_x^r u)= u^n \bar{u}^q D_x^r u,
	\end{equation}
for some $n,q,r \in \N_0$ such that $r=0, \ldots, p-1,$ and $n+q \geq 1$. Here, by convention, $D_x^0 u=u.$ This implies that the nonlinear term contains at least a factor $u$ or one of its derivatives of order less than or equal to $p-1$ and cannot depend only on $\bar{u}.$ Notice that we can include in our analysis several relevant physical models as the ones mentioned in the Introduction. 
Using the same ideas we can also deal with a polynomial nonlinear term given by a linear combination of terms of the form \eqref{nonlinearterm}. 

Before stating the main result of the present section, let us introduce the following auxiliary space.

\begin{Def}\label{X_T} Given $\sigma>1$, let us denote by $2N$ the smallest even number larger than or equal to $\sigma$.  Given $m > 0$ such that $m - \frac{p-1}{2} \in 2\N$, $m> (4N+\frac{11}{2})(p-1)+3$ and considered a real number $\tilde{m} \in \left(\frac{m}{2} + \frac{3}{4}(p-1) + \frac{1}{2},m -(2N+1)(p-1)-p\right]$, we denote by
	$X_T$ the space of all functions $u :[0,T]\times \R \to \C$ such that 
	\beqs\label{defXT}
	\| u\|_{X_T} := \left[\|u(t,\cdot)\|^2_{L^\infty_t H_x^m}\right. &+& \left.\int_0^T\!\!\! \left( \|u(\tau)\|^2_{H^{m+\frac{p-1}2, - \frac{\sigma}2}}+\sum_{k=2}^{p-1}
	\|u(\tau)\|^2_{H^{m+\frac{p-k}2, -\frac{p-k}{2(p-1)}}}\right)d\tau \right.
	\\\nonumber
	&+& \left.\| u \|_{L^\infty_t H_x^{\tilde m, 2N}}^2+\| \partial_tu \|_{L^\infty_t H_x^{\tilde m, 2N}}^2 \right]^{\frac12}<\infty.
	\eeqs
\end{Def}
We shall prove the following result.

\begin{Th}\label{mainNL}
Let $P$ be an operator of the form \eqref{op} satisfying the assumptions of Theorem \ref{smoothing_theorem} and consider the Cauchy problem \eqref{nonlinearCP}.
Then, for every given $m > 0$ such that $m - \frac{p-1}{2} \in 2\N$, $m> (4N+\frac{11}{2})(p-1)+3$ and for every  $\tilde{m} \in \left(\frac{m}{2} + \frac{3}{4}(p-1) + \frac{1}{2},m -(2N+1)(p-1)-p\right]$ we have the following: for every given $g\in \mathscr S(\R)$, the Cauchy problem \eqref{nonlinearCP} admits a unique solution 
$u\in X_{T^*}$ for some $T^* \in (0,T].$
\end{Th}

Let us start by re-writing the right hand side of \eqref{nonlinearCP} in the following way. If $r \geq 1$ we can write
\begin{equation}\label{stella}u^n(t,x)\bar u^q(t,x) D_x^ru (t,x)= \tilde f(t,x,u, \bar u, D_x^ru)+g^n(x)\bar g^q(x)D_x^ru(t,x)
\end{equation}
where
\beqs\label{termine noto}
\tilde f(t,x,u, \bar u, D_x^ru):&=&(u^n(t,x)\bar u^q(t,x)-g^n(x)\bar{g}^q(x))D_x^ru(t,x)
\\\nonumber
&=&\int_0^t\partial_\tau(u^n(\tau,x)\bar u^q(\tau,x))d\tau \cdot D_x^ru(t,x).
\eeqs
Since $g\in\mathscr S(\R)$, the term $g^n\bar g^q D_x^ru$ in \eqref{stella} can be moved on the left hand side of the equation in \eqref{nonlinearCP}, and it trivially fulfills the assumptions for the terms of order $r$. This means that we can write the Cauchy problem \eqref{nonlinearCP} in the following equivalent way:

\begin{equation} \label{nonlinearCP'} \begin{cases}
		\tilde P(t,x,D_t,D_x)u(t,x)= \tilde f(t,x,u,\bar u, D_x^ru) \\ u(0,x)=g(x) \end{cases}, \qquad (t,x) \in [0,T]\times \R,
\end{equation}
where $\tilde P= P-g^n\bar g^q D_x^ru$ satisfies the same assumptions as $P$. Namely, we can apply Theorem \ref{smoothing_theorem} to $\tilde P$.
We can argue similarly if $r=0$, that is $Q(u, \bar u)= u^n \bar{u}^q,$ with $n+q \geq 2, n\geq 1$. Namely, arguing as before we can write
$$
u^n(t,x)\bar u^q(t,x) =\tilde f(t,x,u, \bar u)+ g^{n-1}(x) \bar g^q(x) u(t,x),
$$
where \begin{eqnarray}\label{termine noto2}\tilde{f}(t,x,u ,\bar u)&=&
(u^{n-1}(t,x)\bar u^q(t,x)-g^{n-1}(x)\bar g^q(x) )u(t,x) \\ &=&
\int_0^t\partial_\tau(u^{n-1}(\tau,x)\bar u^q(\tau,x))\, d\tau \cdot u(t,x), \nonumber
\end{eqnarray}
and again we can move $g^{n-1}(x)\bar g^q(x) u(t,x)$ to the linear part of the equation. \\
In what follows we are going to consider only the case when $\tilde f$ is given by \eqref{termine noto}: the case \eqref{termine noto2} works in the same way. Moreover, we are going to denote, for simplicity, $\tilde f(t,x,u,\bar u,D_x^ru)$ by $\tilde f(u)$ when we only need to take care of the dependence on $(u,\bar u, D_x^ru)$, by $\tilde f(t,u)$ or $\tilde f(t,x,u)$ when we need to specify also the dependence on time and/or space variables.\\
Let us so focus, from now on,  on the Cauchy problem \eqref{nonlinearCP'} with $\tilde f(t,u)$ given by \eqref{termine noto}, and consider the map
$$\Phi(u)= W(t,0)g +\int_0^t W(t,\tau) \tilde f(\tau,u)\, d\tau,$$
where, for every $t,\tau \in [0,T]$, $W(t,\tau) h$ denotes the solution of the linear homogeneous Cauchy problem 
\begin{equation} \label{homlinCP} \begin{cases}
		\tilde P(t,x,D_t,D_x)w(t,x)= 0  \\ w(\tau,x)=h(x) \end{cases}, \qquad (t,x) \in [\tau,T]\times \R.
\end{equation}
The idea is to prove that $\Phi$ maps $X_{T^\ast}$ into itself if $T^\ast\leq T$ is small enough, and that it is a contraction on a closed ball of suitable radius in this space, so it admits a unique fixed point $u=\Phi(u)$ which is the unique solution to \eqref{nonlinearCP'}, and so of \eqref{nonlinearCP}. In order to prove this we shall make use of some technical results that we collect in the following subsection.

\begin{Rem}
To better understand the assumptions on $m$,$\tilde m$ in Theorem \ref{mainNL}, we point out the following. The assumption $m - \frac{p-1}{2} \in 2\N$ is needed to apply Leibniz rule in the upcoming Lemma \ref{preliminarylemman'} and in the proof of the main theorem; the lower bound in the choice of $$\tilde{m} \in \left(\frac{m}{2} + \frac{3}{4}(p-1) + \frac{1}{2},m -(2N+1)(p-1)-p\right]$$ is needed in the next Lemma \ref{preliminarylemman'}, as we explain while estimating the term $J_1$, while the upper bound is needed in the proof of the main theorem. The choice of $m> (4N+\frac{11}{2})(p-1)+3$ is needed to guarantee that the interval $\left(\frac{m}{2} + \frac{3}{4}(p-1) + \frac{1}{2},m -(2N+1)(p-1)-p\right]$ is not empty. 
\end{Rem}

\subsection{Some useful estimates}
In the following we shall prove some preliminary results which will be used in the proof of Theorem \ref{mainNL}. In all the next results, given two positive quantities $A$ and $B$, we shall use the notation $A \lesssim B$ if there exists a positive constant $C=C(T)$ and bounded for $T \in (0,1]$ such that $A\leq CB$.
\\
 
This first lemma is a technical result which provides an estimate of the nonlinearity in \eqref{nonlinearCP'} using Sobolev norms in space and $L^\infty$ norms in time.

\begin{Lemma}\label{preliminarylemman'}
Let $u \in X_T$. The following estimate holds:
$$
\int_0^t \|\tilde f(\tau, u) \|_{H^{m-\frac{p-1}{2}, \frac\sigma2}_x}^2 d\tau \leq C_m T\|u\|^{2}_{X_T}(\|u\|^{2(n+q)}_{X_T}+\|g\|^{2(n+q)}_{H^m}).
$$
\end{Lemma}

\begin{proof} Let $\tilde f$ be of the form \eqref{termine noto}. 
First of all we have, denoting $\Delta_{m,p} = m - \frac{p-1}{2}$,
\begin{eqnarray*}\int_0^t \| \tilde{f}(\tau,u) \|_{H^{m-\frac{p-1}{2},\frac\sigma2}_x}^2 d\tau &\lesssim & \sum_{k=0}^{\Delta_{m,p}} 
	\int_0^t \|\langle x \rangle^{\frac{\sigma}{2}}D_x^k \tilde f(\tau,x,u)\|_{L^2_x}^2d\tau
	\\ &\lesssim & \sum_{k=0}^{\Delta_{m,p}} \sum_{j+\ell=k}\frac{k!}{j!\ell!} \int_0^t \int_{\R} \langle x \rangle^{\sigma} |D_x^{j+r} u|^2 |D_x^\ell (u^n\bar u^q-g^n\bar g^q)|^2\, dx d\tau \\
	&=& \sum_{k=0}^{\Delta_{m,p}} \sum_{\stackrel{j+\ell=k}{j \geq \Delta_{m,p}/2}}\frac{k!}{j!\ell!} \int_0^t \int_{\R} \langle x \rangle^{\sigma}|D_x^{j+r} u|^2 |D_x^\ell (u^n\bar u^q-g^n\bar g^q)|^2\, dx d\tau \\
	&&+  \sum_{k=0}^{\Delta_{m,p}
	} \sum_{\stackrel{j+\ell=k}{j<\Delta_{m,p}/2 }}\frac{k!}{j!\ell!} \int_0^t \int_{\R} \langle x \rangle^{\sigma}|D_x^{j+r} u|^2 |D_x^\ell (u^n\bar u^q-g^n\bar g^q)|^2\, dx d\tau 
	\\
	&=:&  J_1+J_2.
\end{eqnarray*} 
Concerning the term $J_1$, using \eqref{termine noto} we have:
$$
j \geq \frac{\Delta_{m,p}}{2}  \implies \ell \leq \frac{1}{2}\Delta_{m,p} = \frac{m}{2} - \frac{p-1}{4},
$$
therefore, using the Sobolev embedding and the definition of the space $X_T$ we get
\begin{align*}
	J_1 &= \sum_{k=0}^{\Delta_{m,p}} \sum_{\stackrel{j+\ell=k}{j \geq \Delta_{m,p}/2}}\frac{k!}{j!\ell!} \int_0^t \int_{\R} |\langle x \rangle^{-\sigma/2}D_x^{j+r} u|^2 \left| \langle x \rangle^{\sigma} \int_0^\tau D^{\ell}_{x}\partial_s\{u^n \overline{u}^{q}\} ds \right|^2dxd\tau \\
	&\lesssim T^2 \sum_{k=0}^{\Delta_{m,p}} \sum_{\stackrel{j+\ell=k}{j \geq \Delta_{m,p}/2}} \sup_{t,x} |\langle x \rangle^{\sigma}D^{\ell}_{x}\partial_s\{u^n \overline{u}^{q}\}|^{2}
	\int_0^t \int_{\R} |\langle x \rangle^{-\sigma/2}D_x^{j+r} u|^2 dxd\tau \\
	&\lesssim T^2 \sum_{k=0}^{\Delta_{m,p}} \sum_{\stackrel{j+\ell=k}{j \geq \Delta_{m,p}/2}}
	\|\langle x \rangle^{\sigma}D^{\ell}_{x}\partial_s\{u^n \overline{u}^{q}\}\|^{2}_{L^\infty_t H^{\frac12+\varepsilon}_x}
	\int_0^t \|\langle x \rangle^{-\frac{\sigma}{2}} D_{x}^{j+r} u\|^2_{L^2_x} d\tau \\
	&\lesssim
	T^2 \|\partial_s\{u^n \overline{u}^{q}\}\|^{2}_{L^\infty_t H_x^{\tilde{m},\sigma}} \int_0^t \| u\|^{2}_{H_x^{m + \frac{p-1}{2},-\frac{\sigma}{2}}} d\tau \\
	&\lesssim
	T^2 \|\partial_s\{u^n \overline{u}^{q}\}\|^{2}_{L^\infty_t H_x^{\tilde{m},2N}} \int_0^t \| u\|^{2}_{H_x^{m + \frac{p-1}{2},-\frac{\sigma}{2}}} d\tau
\end{align*}
since $ \frac{1}{2} + \ell \leq \frac{1}{2} + \frac{m}{2} - \frac{p-1}{4} < \frac{m}{2} + \frac{3(p-1)}{4} + \frac{1}{2} < \tilde{m}$. Due to algebra properties of $H^{\tilde{m},2N}$ we get
$$
J_1 \lesssim T^{2} \|u\|^{2+2(n+q)}_{X_T}.
$$

For $J_2$ we have, since $j + r + \frac12 \leq \Delta_{m,p}/2 + p-1 +\frac12 \leq \frac{m}{2} + \frac{3(p-1)}{4} + \frac12 < \tilde{m}$,

\begin{align*}
	J_2 &= \sum_{k=0}^{\Delta_{m,p}} \sum_{\stackrel{j+\ell=k}{j<\Delta_{m,p}/2}}\frac{k!}{j!\ell!} \int_0^t \int_{\R}|\langle x \rangle^{\frac{\sigma}{2}}D_x^{j+r}u|^2 \cdot |D_x^\ell (u^n\bar u^q-g^n\bar g^q)|^2\, dx d\tau \\
	&\lesssim T  \sum_{k=0}^{\Delta_{m,p}} \sum_{\stackrel{j+\ell=k}{j<\Delta_{m,p}/2}}
	\sup_{t,x} |\langle x \rangle^{\frac{\sigma}{2}}D_x^{j+r}u|^2 \sup_{t} \{\|D_x^\ell(u^n\bar u^q)\|^{2}_{L^2_x} +\|D_x^\ell (g^n\bar g^q)\|^2_{L^2_x} \}
	\\ 
	&\lesssim T \|u\|^{2}_{ L^\infty_t H_x^{\Delta_{m,p}/2+p-1+\frac12+\epsilon,\frac{\sigma}{2}}} (\sup_{t}\|u\|_{H_x^m}^{2(n+q)}+\|g\|_{ H^m}^{2(n+q)}) \\
	&\lesssim T \|u\|^{2}_{L^\infty_t H_x^{\tilde{m},2}} (\sup_{t} \|u\|_{H_x^{m}}^{2(n+q)} + \|g\|_{H^m}^{2(n+q)}) \\
	&\leq T \|u\|^2_{X_T}(\|u\|_{X_T}^{2(n+q)}+\|g\|_{ H^m}^{2(n+q)}).
\end{align*}
\end{proof}

The following lemma is a variant for higher order equations of \cite[Lemma 6.1.3]{KPRV}. 

\begin{Lemma}\label{KenigPonceRV}
	Let $h \in H^{s+n(p-1) , n }$, $n\in\N$. Then, for every $\tau \in [0,T],$ the following estimate holds:
	\begin{equation}\label{stimaKPRV1}
		\sup_{t \in [0,T]} \| W(t,\tau)h\|^{2}_{H^{s, n }} \leq
		 \|h\|^{2}_{H^{s, n}} + \sum_{j=1}^{n} C_{j} T^j \| h\|_{H^{s+j(p-1),n-j}}^2 \ ,
	\end{equation}
	where the positive constants $C_j$ depend on $s,T$ and are bounded as $T\to 0^+$. In particular we have:
	\begin{equation}\label{stimaKPRV2}
		\sup_{t \in [0,T]} \| W(t,\tau)h\|^2_{H^{s, n}} \leq C(1+T^{n})\| h\|^2_{H^{s+n(p-1) , n}}
	\end{equation}
	for some constant $C$ depending on $s, T$ and bounded as $T\to 0^+$.
\end{Lemma}

\begin{proof}
	
	We shall consider $\tau = 0$, because the general case $0 \leq \tau \leq T$ follows similarly. Let then $w = W(t,0)h$ be the solution to \eqref{homlinCP} with $\tau=0$. Since, by assumption, $h \in H^{s+n(p-1)}$, we get from \eqref{energy!} that 
	\begin{equation}\label{seventh_son_of_a_seventh_son}
	\sup_{t\in [0,T]} \|w(t)\|_{H^{s+n(p-1)}} \leq C \|h\|_{H^{s+n(p-1)}}.
	\end{equation}
	We have
	$$
	\tilde P (\langle x \rangle w) = [\tilde P, \langle x \rangle] w \in H^{s+(n-1)(p-1)},
	$$
	because $[\tilde P, \langle x \rangle] \in \textbf{\textrm{SG}}^{p-1,0}$. Hence, since $$\langle x \rangle w(0) =\langle x \rangle h \in H^{s+n(p-1), 1} \subset H^{s+(n-1)(p-1)},$$ we obtain (from the energy inequality and \eqref{seventh_son_of_a_seventh_son})

	\begin{eqnarray}\nonumber
		\|w(t)\|^{2}_{H^{s +(n-1)(p-1), 1}} &=&  \|\langle x \rangle w(t) \|^{2}_{H^{s+(n-1)(p-1)}}
\\\nonumber
&\lesssim& \|\langle x \rangle h\|^{2}_{H^{s+(n-1)(p-1)}} +  \int_{0}^{t} \|[\tilde P, \langle x \rangle] w(\tau)\|^{2}_{H^{s+(n-1)(p-1)}} d\tau  \\\nonumber
		&\lesssim&  \| h\|^{2}_{H^{s+(n-1)(p-1), 1 }} + T \sup_{t \in [0,T]}\|w(t)\|^{2}_{H^{s+n(p-1)}} \\ \label{black'}
		&\lesssim  & \|h\|^{2}_{H^{s+(n-1)(p-1), 1 }} + T \|h\|^{2}_{H^{s+n(p-1)}}.
	\end{eqnarray}
Let us now iterate this procedure for a second time. Since (by \eqref{black'}) $w \in H^{s + (n-1)(p-1), 1}$ and $[\tilde P, \langle x \rangle^{2}] \in \textbf{\textrm{SG}}^{p-1, 1}$, we have
	$
	\tilde P (\langle x \rangle^{2} w) = [\tilde P, \langle x \rangle^{2}] w \in  H^{s+(n-2)(p-1)}.
	$
	Hence, using the fact that $\langle x \rangle^{2}w(0) = \langle x \rangle^{2}h \in H^{s+n(p-1),2} \subset H^{s+n(p-1)}$, the energy inequality \eqref{energy!} applied to $\langle x \rangle^{2} w$ gives
	\begin{eqnarray*}
		\|w(t)\|^{2}_{H^{s+(n-2)(p-1), 2}} &= & \|\langle x \rangle^{2} w(t) \|^{2}_{H^{s+(n-2)(p-1)}} \\
		&\lesssim&  \|\langle x \rangle^{2} h\|^{2}_{H^{s+(n-2)(p-1)}} +  \int_{0}^{t} \|[\tilde P, \langle x \rangle^{2}] w(\tau)\|^{2}_{H^{s+(n-2)(p-1)}} d\tau  \\
&\lesssim&  \| h\|^{2}_{H^{s+(n-2)(p-1),2 }} + T \sup_{t \in [0,T]}\|w(t)\|^{2}_{H^{s+(n-1)(p-1),1}}
\\
&\lesssim& \| h\|^{2}_{H^{s+(n-2)(p-1),2 }} + T\|h\|^{2}_{H^{s+(n-1)(p-1), 1 }} + T^2 \|h\|^{2}_{H^{s+n(p-1)}}.
	\end{eqnarray*}
Iterating the above argument $j\geq 2$ times we get
\beqsn\|w(t)\|^{2}_{H^{s+(n-j)(p-1), j}} &=&  \|\langle x \rangle^{j} w(t) \|^{2}_{H^{s+(n-j)(p-1)}}
\\
&\lesssim &\| h\|^{2}_{H^{s+(n-j)(p-1),j }} + T\|h\|^{2}_{H^{s+(n-j+1)(p-1), j-1 }}
\\
&& +T^2\|h\|^{2}_{H^{s+(n-j+2)(p-1), j-2}}+... + T^j \|h\|^{2}_{H^{s+n(p-1)}}
\eeqsn
and for $j=n$ we obtain \eqref{stimaKPRV1}.
The estimate \eqref{stimaKPRV2} follows directly from \eqref{stimaKPRV1}.
\end{proof}
Finally, in what follows we shall also make use of the following lemma:

\medskip

{\bf{Lemma 6.0.1 in \cite{FS}}}\label{FS_Schauder}
{\sl{
Let $f,g \in H^s(\R^n), s >n/2,$ such that $f, g \in H^{n/2+\varepsilon,2N}(\R^n)$ for some $\varepsilon >0, N\in\N$. Then $fg \in H^{s,2N}(\R^n)$ and
$$\| fg \|^2_{H^{s,2N}} \lesssim \|f\|^2_{H^{n/2+\varepsilon, 2N}}\|g\|_{H^s}^2 + \|g\|^2_{H^{n/2+\varepsilon,2N}}\|f\|_{H^s}^2.$$ }}

\subsection{Proof of Theorem \ref{mainNL}}

Let us fix $u$ in a ball of radius $R>0$ in $X_T$ and let $v=\Phi(u)$ be the solution of the linear Cauchy problem
\begin{equation} \label{linearizedCP} \begin{cases}
		\tilde P(t,x,D_t,D_x)v(t,x)= \tilde f(t,u)  \\ v(\tau,x)=g(x) \end{cases}, \qquad (t,x) \in [\tau,T]\times \R.
\end{equation}
Let us first check that $v\in X_T$. By Lemma \ref{preliminarylemman'} we have $\tilde{f} \in L^{2}([0,T]; H^{\tilde{m} - \frac{p-1}{2}, \frac{\sigma}{2}})$ for any $u \in X_T$, so the linear smoothing estimate \eqref{energy!!} of Theorem \ref{smoothing_theorem} implies
\beqs
\|v(t)\|_{H^{m}}^2 + \int_0^t \left( \| v(\tau)\|^2_{H_x^{m+\frac{p-1}2, -\frac\sigma2}}\right.&+&\left.\sum_{k=2}^{p-1}\| v(\tau)\|^2_{H_x^{m+\frac{p-k}2, -\frac{p-k}{2(p-1)}}} \right)d\tau 
\\\nonumber
&\leq& C_{m,T} \left(\|g \|_{H^{m}}^2+ \int_0^t \| \tilde f(\tau,u)\|^2_{H_x^{m-\frac{p-1}{2},\frac{\sigma}{2}}}d\tau\right).
\eeqs
By Lemma \ref{preliminarylemman'} we get for every $t \in [0,T]$:
\beqsn
\|v(t)\|_{H^{m}}^2 +\int_0^t\left( \| v(\tau)\|^2_{H_x^{m+\frac{p-1}2, -\frac\sigma2}}\right.&+&\left.\sum_{k=2}^{p-1}\| v(\tau)\|^2_{H_x^{m+\frac{p-k}2, -\frac{p-k}{2(p-1)}}} \right)d\tau \\
&\leq &C_{m,T} \left(\|g \|_{H^{m}}^2+  T\|u\|^{2n}_{X_T}(\|u\|^{2(n+q)}_{X_T}+\|g\|^{2(n+q)}_{H^m}) \right) \\ 
&\leq & C_{m,T} \left(\|g \|_{H^{m}}^{2}+ T\|u\|^{4n+2q}_{X_T}+T\|u\|^{2n}_{X_T}\|g\|^{2(n+q)}_{H^m}\right)<\infty 
\eeqsn
where $C=C_{m,T}>0$ is bounded for $T\to 0^+$.  
\\
By definition of the space $X_T$ it remains to estimate
 \beqsn
\|v\|_{L^\infty_t H_x^{\tilde{m},2N}}^2+\|\partial_tv\|_{L^\infty_t H_x^{\tilde{m},2N}}^2&=&\|v\|_{L^\infty_t H_x^{\tilde{m},2N}}^2+\|i\tilde f(u) - \ds\sum_{j=0}^p ia_jD_x^jv\|_{L^\infty_t H_x^{\tilde{m},2N}}^2
\\
&\leq& C \left( \|\tilde f(u)\|_{L^\infty_t H_x^{\tilde{m},2N}} ^2+ \|v\|_{L^\infty_t H_x^{\tilde{m}+p,2N}}^2 \right).
\eeqsn
For the first term in the right hand side 
we have, since $\tilde{m} + p-1 \leq m$:

\beqsn
\|\tilde f(t,u)\|_{H_x^{\tilde{m},2N }} &=&\|\langle \cdot \rangle^{2N}D_x^ru(t,\cdot)\int_0^t\partial_s(u^n(s,\cdot)\bar u^q(s,\cdot))ds\|_{H_x^{\tilde{m}}} \leq 
\\&\leq & \|D_x^ru\|_{L^\infty_t H^{\tilde{m}}}\int_0^t \|  \langle \cdot \rangle^{2N}\partial_s(u^n(s,\cdot)\bar u^q(s,\cdot)) \|_{H_x^{\tilde{m}}}ds
\\&\leq &\|u\|_{L^\infty_t H_x^{\tilde m+r}}\int_0^t \|  n\langle \cdot \rangle^{2N}\partial_su(s,\cdot)u^{n-1}(s,\cdot)\bar u^q(s,\cdot)) \|_{H_x^{\tilde{m}}}ds
\\
&+& \|u\|_{L^\infty_t H_x^{\tilde{m}+r}} \int_0^t \|  q\langle \cdot \rangle^{2N}\partial_s \bar u (s,\cdot)u^n(s,\cdot)\bar u^{q-1}(s,\cdot) \|_{H_x^{\tilde{m}}}ds
\\
&\lesssim& T \|u\|_{L^\infty_t H_x^{\tilde{m}+p-1}} \sup_{t} \|u\|_{H_x^{\tilde m}}^{n+q-1} \sup_{t} \|\partial_t u\|_{H_x^{\tilde m,2N}}
\\
&\lesssim  &T \|u\|_{X_T}^{n+q+1}.
\eeqsn
As for the second one, we notice that for every $t \in [0,T]$ we have
\begin{multline*}
	\| v (t)\|_{ H_x^{\tilde{m}+p,2N}}^2 = \left\| W(t,0)g+ \int_0^t W(t,\tau) \tilde f(\tau, u) \, d\tau \right\|^2_{H_x^{\tilde{m}+p,2N}}\\
	\leq 2 \left[  \left\| W(t,0)g \right\|^2_{H_x^{\tilde{m}+p,2N}} + \int_0^t \left\| W(t,\tau) \tilde f(\tau, u)\right\|^2_{H_x^{\tilde{m}+p,2N}}\, d\tau \right].
\end{multline*}

Applying Lemma \ref{KenigPonceRV} with $s=\tilde{m}+p$ and $n=2N$, then Lemma 6.0.1 in \cite{FS} to the last line of the inequality above, we obtain 
\begin{multline*}
	\| v (t)\|_{ H_x^{\tilde{m}+p,2N}}^2  \leq 2   C_{m,T}(1+T^{2N}) \left[\| g\|_{H^{\tilde{m}+p+2N(p-1), 2N}}^2 +T^2\| \tilde f(t,u)\|^2_{L^\infty_t H_x^{\tilde{m}+p+2N(p-1),2N}}\right] \\
	\leq  2   C_{m,T}(1+T^{2N})
	\left[ \|g\|_{H^{\tilde{m}+p+2N(p-1), 2N}}^2 + T \| u^n\bar u^q-g^n\bar g^q\|_{L^\infty_t H_x^{\frac12+\varepsilon, 2N}}^2 \cdot \|D_x^ru\|^2_{L^\infty_t H_x^{\tilde{m}+p+2N(p-1)}} \right.
	\\
	\left.\hskip+6.6cm+ T\| D_x^ru\|_{L^\infty_t H_x^{\frac12+\varepsilon, 2N}}^2 \cdot \|u^n\bar u^q-g^n\bar g^q\|^2_{L^\infty_t H_x^{\tilde{m}+p+2N(p-1)}}  \right].
	\end{multline*} 
 Finally, we observe that $\tilde m> \frac{m}{2} + \frac{3}{4}(p-1) + \frac{1}{2}> (p-1) + \frac{1}{2} > \frac{1}{2} + \varepsilon$, so both $ \| u\|_{H_x^{\frac12+\varepsilon, 2N}}$ and $ \| D_x^r u\|_{H_x^{\frac12+\varepsilon, 2N}}$ can be bounded from above by $ \| u\|_{H_x^{\tilde m, 2N}}$. Moreover,  $\tilde m+p+2N(p-1)+(p-1) \leq m$ by assumption $\tilde m\leq m -(2N+1)(p-1)-p$, so both $ \|u\|^2_{H_x^{\tilde{m}+p+2N(p-1)}} $ and $\| D_x^ru\|^2_{H_x^{\tilde{m}+p+2N(p-1)}} $ with $r\leq p-1$ can be bounded from above by $ \|u\|^2_{H_x^{m}}.$ Hence, using (8) of Proposition \ref{wSobolevproperties}, we conclude that
$$ \| v (t)\|_{ H_x^{\tilde{m}+p,2N}}^2\lesssim (1+T^{2N}) \left[ \|g\|_{H^{m, 2N}}^2 + 2T\|u\|^{2}_{X_T}\left(\|u\|^{2(n+q)}_{X_T}+\|g\|_{H^{m, 2N}}^{2(n+q)}\right) \right]<\infty.
$$
In conclusion, we obtain that there exists a constant $C_{m}>0$ depending on $T$ and bounded for $T \to 0^+$ such that

\beqsn 
\|v \|_{X_T}^2 &\leq &C_{m} \left\{\|g \|_{H^{m}}^{2}+T\|u\|^{2n}_{X_T}\|g\|^{2(n+q)}_{H^m}+ T\|u\|^{4n+2q}_{X_T}
+T \|u\|_{X_T}^{2(n+q+1)}\right.
\\
&+&\left.(1+T^{2N}) \left[ \|g\|_{H^{m,2N}}^2 + T\|u\|^{2}_{X_T}\left(\|u\|^{2(n+q)}_{X_T}+\|g\|_{H^{m, 2N}}^{2(n+q)}\right) \right]\right\}<\infty
\eeqsn
which imply that $v\in X_T$. Now, choosing $R^2=2C_m \left\{\|g \|_{H^{m}}^{2} +\|g\|_{H^{m, 2N}}^2\right)$ we have

\beqsn
\|v \|_{X_T} ^2&\leq& \frac {R^2}2+ C_m T\left[\|u\|^{2n}_{X_T}\|g\|^{2(n+q)}_{H^m}+ \|u\|^{4n+2q}_{X_T}
+ \|u\|_{X_T}^{2(n+q+1)}+T^{2N}\|g\|^2_{H^{m, 2N}}\right. \\
&&\left.+(1+T^{2N})\|u\|^{2}_{X_T}\left(\|u\|^{2n}_{X_T}+\|g\|_{H^{m, 2N}}^{2n}\right) 
\right]
\\
&\leq &
\frac {R^2}2+ C'_m T\left[\|u\|^{2n}_{X_T}R^{2(n+q)}+ \|u\|^{4n+2q}_{X_T}
+ \|u\|_{X_T}^{2(n+q+1)}+T^{2N}R^2\right. \\
&&\left.+(1+T^{2N})\|u\|^{2}_{X_T}\left(\|u\|^{2n}_{X_T}+R^{2n}\right) 
\right] \\
&\leq &
\frac {R^2}2+ C'_m T\left[R^{4n+2q}
+ R^{2(n+q+1)}+T^{2N}R^2+(1+T^{2N})R^{2n+2} 
\right],
\eeqsn
for a new constant $C'_m>0$ depending on $T$ and bounded for $T \to 0^+$. 
Hence, we can choose $T$ is sufficiently small such that $\|v \|_{X_T} ^2\leq R^2$, that is we obtain that $\Phi$ maps the closed ball of radius $R$ in $X_T$ into itself.

Let us now show that $\Phi$ is a contraction on $B_R =\{u \in X_T: \|u\|_{X_T}\leq R\}.$ Fix $ u_1, u_2 \in B_R$ and let $v_1=\Phi(u_1), v_2 =\Phi(u_2).$ The function $v_1- v_2$ solves
\begin{equation}\label{contractionCP}
	\begin{cases} \tilde P(v_1- v_2)=\tilde f(t,u_1)- \tilde f(t, u_2) \\ (u_1- u_2)(0,x)=0 \end{cases} \qquad (t,x) \in [0,T] \times \R.
\end{equation}
By Theorem \ref{smoothing_theorem} and the definition of $\|-\|_{X_T}$we have
\begin{equation}\label{differenceenergyestimate}
	\|v_1- v_2 \|^2_{X_T} \leq C_{m,T} \int_0^T \| \tilde f(\tau,u_1)- \tilde f(\tau,u_2) \|^2_{H_x^{m-\frac{p-1}{2}, \frac{\sigma}{2}}}  d\tau + \|v_1- v_2 \|_{L^\infty_t H_x^{\tilde m,2N}}^2+ \|\partial_t(v_1- v_2) \|_{L^\infty_t H_x^{\tilde m, 2N}}^2.
\end{equation}
Let us estimate the three terms in the right-hand side of \eqref{differenceenergyestimate}.
We first notice that we may write
\begin{eqnarray}\nonumber
	\tilde f(u_1)- \tilde f(u_2) &=& (u_1^n \bar{u}_1^q - g^n \bar g^q) D_x^r u_1 - (u_2^n \bar{u}_2^q - g^n \bar g^q)D_x^r u_2 \\ \nonumber
	&=& (u_1^n-u_2^n) \bar u_1^q D_x^r u_1+ (\bar u_1^q - \bar u_2^q) u_2^n D_x^r u_1+ (u_2^n\bar u_1^q - g^n \bar g^q ) D_x^r (u_1-u_2) \\ \nonumber
	&=& (u_1-u_2)P_{n-1}(u_1,u_2) \bar u_1^q D_x^r u_1+ (\bar u_1 - \bar u_2)P_{q-1}(u_1,u_2) u_2^n D_x^r u_1 \\ \label{splitt} &&+ (u_2^n\bar u_1^q - g^n \bar g^q ) D_x^r (u_1-u_2),
\end{eqnarray}
for some suitable polynomials $P_{n-1}, P_{q-1}$ of degree $n-1$ and $q-1$ respectively.
Notice that we can write
\begin{equation}\label{versioneconintegrale3}
	u_2^n\bar u_1^q - g^n \bar g^q = \int_0^t \partial_s (u_2^n\bar u_1^q(s, \cdot ))\, ds.
\end{equation}


Now we have (recalling $\Delta_{m,p} = m-\frac{p-1}{2}$)

\begin{eqnarray*}
\int_0^T\hskip-0.6cm& \|&\hskip-0.4cm (u_1- u_2)P_{n-1}(u_1, u_2)\bar u_1^q D_x^ru_1\|^2_{H_x^{m-\frac{p-1}{2}, \frac{\sigma}{2}}} d\tau
\\
&=&\sum_{k=0}^{\Delta_{m,p}}\int_0^T \|\langle x \rangle^{\frac{\sigma}{2}}D_x^k\left( (u_1- u_2)P_{n-1}(u_1, u_2)\bar u_1^q D_x^ru_1\right)\|^2_{L^2_x} d\tau
\\
&=&\sum_{k=0}^{\Delta_{m,p}} \sum_{\stackrel{k_1+k_2+k_3+k_4=k}{k_4 \geq \frac{\Delta_m,p}{2} }}\hskip-0.8cmC_{k_1,k_2,k_3,k_4} \\ &&\cdot \int_0^T\!\!\!\int_\R |\langle x\rangle^{\sigma}D_x^{k_1} (u_1- u_2)|^2|D_x^{k_2}P_{n-1}(u_1, u_2)|^2|D_x^{k_3}\bar u_1^q|^2 |\langle x\rangle^{-\sigma/2}D_x^{k_4+r}u_1|^2 dx d\tau
\\
&+& \sum_{k=0}^{\Delta_{m,p}}\sum_{\stackrel{\tilde k+k_4=k}{k_4 < \frac{\Delta_{m,p}}{2}}}\hskip-0.3cmC_{\tilde k,k_4} \int_0^T\!\!\!\int_\R |\langle x\rangle^{-\sigma/2} D_x^{\tilde k} \left((u_1- u_2)P_{n-1}(u_1, u_2)\bar u_1^q\right)|^2 \cdot |\langle x\rangle^{\sigma}D_x^{k_4+r}u_1|^2 dx d\tau
\\&=& : K_1+K_2.
\end{eqnarray*}
Now, in $K_1$ we have
\begin{eqnarray*}\sup_{(t,x) \in [0,T]\times \R}|D_x^{k_2}P_{n-1}(u_1, u_2)|^2|D_x^{k_3}\bar u_1^q|^2 &\lesssim&  \|u_1 \|_{L^\infty_t H_x^m}^{2q}(\|u_1 \|^2_{L^\infty_t H_x^m}+\|u_2 \|^2_{L^\infty_t H_x^m})^{n-1} \\ &\lesssim&  (\|u_1\|_{X_T}+\|u_2\|_{X_T})^{2(n+q-1)},
\end{eqnarray*}
so this supremum can be moved outside the integral. Moreover, if $k_4 \geq \Delta_{m,p}/2,$ then on one hand $k_1+\frac{1}{2}<\frac{\Delta_{m,p}}{2} + \frac{1}{2} < \tilde m,$ and on the other hand $k_4+r\leq  \Delta_{m,p}+p-1 = m + \frac{p-1}{2}$; then, writing 
\begin{multline*}
\int_0^T\!\!\!\int_\R |\langle x\rangle^{\sigma}D_x^{k_1} (u_1- u_2)|^2 |\langle x\rangle^{-\sigma/2}D_x^{k_4+r}u_1|^2 dx d\tau \\ =  \int_0^T\!\!\!\int_\R \left|\int_0^\tau \langle x\rangle^{\sigma}D_x^{k_1}\partial_s (u_1- u_2)\, ds \right|^2  |\langle x\rangle^{-\sigma/2}D_x^{k_4+r}u_1|^2 dx d\tau 
\end{multline*}
and using the fact that
\begin{eqnarray*} \left|\int_0^\tau  \langle x\rangle^{\sigma}D_x^{k_1}\partial_s (u_1- u_2)\, ds \right| &\leq& T \sup_{(t,x)\in [0,T]\times \R} | \langle x\rangle^{\sigma}D_x^{k_1}\partial_t (u_1- u_2) | \\ &\leq& T \|\partial_t(u_1-u_2)\|_{L^\infty_t H_x^{\tilde{m}, \sigma}} \leq T \| u_1-u_2\|_{X_T},
\end{eqnarray*}
and $$\int_0^T \int_\R |\langle x\rangle^{-\sigma/2}D_x^{k_4+r}u_1|^2 dxd\tau \leq \int_0^T \|u_1 \|^2_{H_x^{m+\frac{p-1}2,-\frac\sigma 2}}\, dt \leq \|u_1\|_{X_T}^2, $$ we conclude that 
$$
K_1 \leq C'_m T^2 \|u_1\|_{X_T}^2 (\|u_1\|_{X_T}+\|u_2\|_{X_T})^{2(n+q-1)} \| u_1-u_2\|^2_{X_T}.
$$
To estimate $K_2$ we observe that
if $k_4 <\Delta_{m,p}/2,$ then $k_4+r + \frac{1}{2} < \frac{\Delta_{m,p}}{2} + p-1 + \frac{1}{2} < \tilde m$; therefore, using the algebra properties of $H^m(\R)$, we get 

\beqsn
K_2&\lesssim & \|u_1\|^2_{L^\infty_t H_x^{\tilde{m}, \sigma}} \int_0^T \|(u_1- u_2)P_{n-1}(u_1, u_2)\bar u_1^q\|_{H_x^{m-\frac{p-1}{2}, -\frac{\sigma}{2} }}^2 d\tau
\\
&\lesssim & \|u_1\|^2_{L^\infty_t H_x^{\tilde{m}, \sigma}} \int_0^T \|(u_1- u_2)P_{n-1}(u_1, u_2)\bar u_1^q\|_{H_x^{m }}^2 d\tau
\\
&\lesssim & T \|u_1\|^2_{L^\infty_t H_x^{\tilde{m}, \sigma}}\|u_1-u_2\|^2_{L^\infty_t H_x^{m}}(\|u_1\|_{L^\infty_t H_x^{m}}+\|u_2\|_{L^\infty_t H_x^{m}})^{2(n-1)}\|u_1\|^{2q}_{L^\infty_t H_x^{m}}
\\
&\lesssim & T\|u_1\|_{X_T}^2 (\|u_1\|_{X_T}+\|u_2\|_{X_T})^{2(n+q-1)} \| u_1-u_2\|^2_{X_T}.
\eeqsn

Hence, we obtain, for a positive constant $C_m$, 
\beqsn \int_0^T \hskip-0.3cm\left\|\left((u_1-u_2)P_{n-1}(u_1,u_2) \bar u_1^q D_x^r u_1\right)\right.&&\hskip-0.6cm\left.(\tau) \right\| ^2_{H_x^{m-\frac{p-1}{2}, \frac{\sigma}{2} }}d\tau
\\
&\leq& C_m T  \|u_1\|_{X_T}^2 (\|u_1\|_{X_T}+\|u_2\|_{X_T})^{2(n+q-1)} \| u_1-u_2\|^2_{X_T}.\eeqsn
Repeating readily the same argument we obtain also
\beqsn \int_0^T \hskip-0.3cm \left\|\left((\bar u_1-\bar u_2)P_{q-1}(u_1,u_2)  u_2^n D_x^r u_1\right)\right.&&\hskip-0.6cm \left.(\tau) \right\| ^2_{H_x^{m - \frac{p-1}{2}, \frac{\sigma}{2}}}d\tau
\\
& \leq& C_m T  \|u_1\|_{X_T}^2 (\|u_1\|_{X_T}+\|u_2\|_{X_T})^{2(n+q-1)} \| u_1-u_2\|^2_{X_T}.\eeqsn
Concerning the third term in the right hand side of \eqref{splitt}, using \eqref{versioneconintegrale3}, we have
\beqsn\int_0^T&& \left\| (u_2^n\bar u_1^q - g^n \bar g^q ) D_x^r (u_1-u_2)\right\|_{H_x^{m-\frac{p-1}{2}, \frac{\sigma}{2} }}^2 \, d\tau
\\& \leq& \sum_{k=0}^{\Delta_{m,p}} \sum_{\stackrel{k_1+k_2 =k}{k_2 \geq \frac{\Delta_{m,p}}{2}}} \int_0^T \int_\R \left|\int_0^\tau \langle x \rangle^{\sigma}D_x^{k_1}\partial_s(u_2^n\bar u_1^q)(s, x)\, ds \right|^2 \cdot |\langle x \rangle^{-\sigma/2}D_x^{k_2+r}(u_1-u_2)|^2 \, dx d\tau \\
&+& \sum_{k=0}^{\Delta_{m,p}} \sum_{
\stackrel{k_1+k_2 =k}{k_2 < \frac{\Delta_{m,p}}{2}}} \int_0^T \int_\R \left| \langle x \rangle^{-\sigma/2} D_x^{k_1}(u_2^n\bar u_1^q(s, x)-g^n(x)\bar g^q(x)) \right|^2 \cdot |\langle x \rangle^{\sigma}D_x^{k_2+r}(u_1-u_2)|^2 \, dx d\tau
\\
&=:& \widetilde{K}_1+\widetilde{K}_2.
\eeqsn
Since for $k_2 \geq \frac{\Delta_{m,p}}{2}$ we have $k_1 + \frac{1}{2} \leq \frac{\Delta_{m,p}}{2} + \frac{1}{2} < \tilde{m}$, then we get (recalling that we are assuming $\sigma<2$)
\begin{eqnarray*}\left|\int_0^\tau \langle \cdot \rangle^{\sigma}D_x^{k_1}\partial_s(u_2^n\bar u_1^q)(s, \cdot)\, ds \right| &\leq& T \sup_{(t,x) \in [0,T]\times \R} |\langle x \rangle^{\sigma} D_x^{k_1}\partial_t (u_2^n\bar u_1^q)| \\ &\lesssim& T \|\partial_t (u_2^n\bar u_1^q)\|_{L^\infty_t H_x^{\tilde m, \sigma}}
\\
&\lesssim &  T\|\partial_t u_2 \|_{L^\infty_t H_x^{\tilde m, \sigma}} \| u_2 \|_{L^\infty_t H_x^{\tilde m, \sigma}}^{n-1}\|u_1\|_{L^\infty_t H_x^{\tilde m, \sigma}}^q
\\
&+& T\|\partial_t u_1 \|_{L^\infty_t H_x^{\tilde m, \sigma}}\| u_2 \|_{L^\infty_t H_x^{\tilde m, \sigma}}^{n} \| u_1 \|_{L^\infty_t H_x^{\tilde m, \sigma}}^{q-1} \\ &\lesssim& T( \|u_2\|_{X_T}^{n}+ \|u_1\|_{X_T}^{q}).
\end{eqnarray*}

Then, we conclude that
\beqsn
\widetilde{K}_1 &\lesssim&  T^2( \|u_2\|_{X_T}+ \|u_1\|_{X_T})^{2(n+q)} \int_0^T \| \langle x\rangle^{-\sigma/2}(u_1-u_2)(t)\|^2_{H_x^{m+\frac{p-1}{2}}}\, dt 
\\
&\lesssim&  T^2 ( \|u_2\|_{X_T}+ \|u_1\|_{X_T})^{2(n+q)} \cdot \| u_1-u_2 \|_{X_T}^2.
\eeqsn
Concerning $\widetilde{K}_2$, we have $k_2+r+\frac{1}{2}< \frac{\Delta_{m,p}}{2} + p-1 + \frac{1}{2} < \tilde m$, so
\beqsn
\tilde K_2&\lesssim&\|u_1-u_2\|_{L^\infty_t H_x^{\tilde m, \sigma}}^2\int_0^T  \|  u_2^n\bar u_1^q(s)-g^n\bar g^q\|_{H^m_x}^2 d\tau
\\
&\lesssim& T\| u_1-u_2 \|_{X_T}^2( \|u_2\|_{X_T}+ \|u_1\|_{X_T}+  \|g\|_{X_T})^{2(n+q)}.
\eeqsn

Summing up, we obtain the following estimate:
\begin{equation}\label{stimaintegraledellanorma} \int_0^T \| \tilde f(\tau,u_1)- \tilde f(\tau,u_2) \|^2_{H^{m-\frac{p-1}{2}, \frac{\sigma}{2} }}  d\tau \lesssim T (\|u_1\|_{X_T}+\|u_2\|_{X_T}+\|g\|_{X_T} )^{2n+2q} \|u_1-u_2\|^2_{X_T}.
\end{equation}

For the estimate of the last two terms in \eqref{differenceenergyestimate}, arguing as before we write
\beqsn
\|v_1- v_2\|_{L^\infty_t H_x^{\tilde{m}, 2N}}^2+\|\partial_t(v_1- v_2)\|_{L^\infty_t H_x^{\tilde{m},2N}}^2
&\leq& \|v_1- v_2\|_{L^\infty_t H_x^{\tilde{m},2N}}^2+\|\tilde f(u_1)-\tilde f(u_2)\|_{L^\infty_t H_x^{\tilde{m},2N}}^2 
\\
&+&\| \ds\sum_{j=0}^p ia_jD_x^j(v_1- v_2)\|_{L^\infty_t H_x^{\tilde{m},2N}}^2
\\
&\lesssim& \left( \|\tilde f(u_1)-\tilde f( u_2)\|_{L^\infty_t H_x^{\tilde{m},2N}} ^2+ \|v_1- v_2\|_{L^\infty_t H_x^{\tilde{m}+p,2N}}^2 \right).
\eeqsn 
To estimate $ \|\tilde f(u_1)-\tilde f( u_2)\|_{L^\infty_t H_x^{\tilde{m},2N}} ^2$ we can use again \eqref{splitt}. Namely, writing
$$
(u_1-u_2)P_{n-1}(u_1,u_2)\bar u_1^q D_x^r u_1 = \int_0^t \partial_s (u_1-u_2)(s,x)\, ds \cdot P_{n-1}(u_1,u_2)\bar u_1^q D_x^r u_1,
$$
we have, by Minkowski inequality and the algebra properties of $H^{\tilde{m}}(\R)$: 
\begin{eqnarray*}
\|(u_1\hskip-0.3cm&-&\hskip-0.3cmu_2)P_{n-1}(u_1,u_2)\bar u_1^q D_x^r u_1\|_{H_x^{\tilde m, 2N}} =
\|\langle \cdot \rangle^{2N}(u_1-u_2)P_{n-1}(u_1,u_2)\bar u_1^q D_x^r u_1\|_{H_x^{\tilde m}}\\
&\leq& \|\langle \cdot \rangle^{2N}(u_1-u_2)\|_{H_x^{\tilde m}} \cdot \|P_{n-1}(u_1,u_2) \|_{H_x^{\tilde m}}\cdot \|\bar u_1^q \|_{H_x^{\tilde m}}\cdot \| D_x^r u_1\|_{H_x^{\tilde m}} \\
&\leq&  \int_0^t \| \langle \cdot \rangle^{2N}\partial_t(u_1-u_2)\|_{H_x^{\tilde m}} \, d\tau  \cdot \|P_{n-1}(u_1,u_2)\|_{L^\infty_t H_x^{\tilde{m}}}\cdot\| u_1\|_{L^\infty_t H_x^{\tilde{m}}}^q \cdot \|u_1\|_{L^\infty_t H_x^{\tilde{m}+r}} \\
&\leq& T \|\partial_t(u_1-u_2)\|_{L^\infty_t H_x^{\tilde{m}, 2N}} \cdot \|P_{n-1}(u_1,u_2)\|_{L^\infty_t H_x^{\tilde{m}}}\cdot\| u_1\|_{L^\infty_t H_x^{\tilde{m}}}^q \cdot \|u_1\|_{L^\infty_t H_x^{\tilde{m}+r}}
\\ 
&\leq&	T \|u_1-u_2\|_{X_T} (\|u_1\|_{X_T}+\|u_2\|_{X_T})^{n-1} \|u_1\|_{X_T}^q \| u_1\|_{X_T}, 
\end{eqnarray*}
where we have used the fact that $\tilde m+r  \leq \tilde m + p-1 < m.$
Similarly, we obtain 
$$
\| (\bar u_1- \bar u_2)P_{q-1}(u_1,u_2)u_2^n D_x^r u_1\|_{L^\infty_t H_x^{\tilde m, 2N }} \leq T \|u_1-u_2 \|_{X_T} \|u_2\|_{X_T}^n (\|u_1\|_{X_T} + \|u_2\|_{X_T})^{q-1} \| u_1\|_{X_T}.
$$ 
Moreover, using \eqref{versioneconintegrale3} and Minkowski inequality again, we get
\begin{eqnarray*}
\| (u_2^n\bar u_1^q \hskip-0.2cm&-& \hskip-0.2cmg^n \bar g^q ) D_x^r (u_1-u_2)\|_{H_x^{\tilde m, 2N}} = \left\|\int_0^t \partial_s (u_2^n \bar u_1^q) D_x^r (u_1-u_2)ds\right\|_{H_x^{\tilde m, 2N}} \\ &\lesssim& \int_0^T \|\partial_s (u_2^n \bar u_1^q)  D_x^r (u_1-u_2)\|_{H_x^{\tilde m, 2N}} ds\\ &\lesssim& T\left( \| \partial_t u_2 \|_{L^\infty_t H_x^{\tilde m, 2N}} \| u_2\|^{n-1}_{L^\infty_t H_x^{\tilde m}}\| u_1\|^{q}_{L^\infty_t H_x^{\tilde m}}+\| \partial_t u_1 \|_{L^\infty_t H_x^{\tilde m, 2N}} \| u_1\|^{q-1}_{L^\infty_t H_x^{\tilde m}}\| u_2\|^{n}_{L^\infty_t H_x^{\tilde m}} \right)
\\
&&\cdot \| u_1-u_2 \|_{L^\infty_t H_x^{m}} \\ &\lesssim&  T( \|u_2\|_{X_T}+ \|u_1\|_{X_T})^{n+q} \| u_1-u_2 \|_{X_T}. 
\end{eqnarray*}
We finally obtain 
\begin{equation}\label{secondastima}
\| \tilde{f}(u_1)-\tilde f (u_2)\|_{L^\infty_t H_x^{\tilde m, 2N}} \lesssim T (\| u_1\|_{X_T}+ \| u_2\|_{X_T})^{n+q} \| u_1-u_2 \|_{X_T}.
\end{equation}

Moreover, taking into account that $v_1- v_2$ solves \eqref{contractionCP}, we have

\begin{eqnarray}\label{square}
\| v_1- v_2 \|^2_{H_x^{\tilde m+p,2N}} &=& \left\| \langle \cdot \rangle^{2N} \int_0^t W(t,\tau)(\tilde{f}(\tau, u_1)-\tilde{f}(\tau, u_2)\,d\tau\right\|^2_{H_x^{\tilde m +p}} \, \\ \nonumber
&\leq &  \int_0^t \left\| \langle \cdot \rangle^{2N} W(t,\tau)(u_1-u_2)P_{n-1}(u_1,u_2) \bar u_1^q D_x^r u_1\right\|_{H_x^{\tilde m +p}}^2\, d\tau 
\\ \nonumber
&&+\int_0^t \left\| \langle \cdot \rangle^{2N}  W(t,\tau)u_2^n(\bar u_1-\bar u_2)P_{q-1}(\bar u_1, \bar u_2)D_x^r u_1\right\|_{H_x^{\tilde m +p}}^2\, d\tau \\ \nonumber
&&+  \int_0^t \left\| \langle \cdot \rangle^{2N} W(t,\tau)u_2^n \bar{u}_1^q D_x^r (u_1-u_2)\right\|_{H_x^{\tilde m +p}}^2\, d\tau.
\end{eqnarray}
Applying Lemma \ref{KenigPonceRV}, precisely \eqref{stimaKPRV2}, to each term in the right-hand side we obtain

\begin{eqnarray} \label{brutta} 	
\| v_1- v_2 \|^2_{H_x^{\tilde m+p, 2N}} &\lesssim&  T(1+T^{2N}) \left\| (u_1-u_2)P_{n-1}(u_1,u_2) \bar u_1^q D_x^r u_1\right\|_{H_x^{\tilde m +p+2N(p-1), 2N}}^2 \\ && \nonumber +   T(1+T^{2N}) \left\|u_2^n(\bar u_1-\bar u_2)P_{q-1}(\bar u_1, \bar u_2)D_x^r u_1 \right\|_{H_x^{\tilde m +p+2N(p-1), 2N}}^2 \\ &&+ 
\nonumber	 T(1+T^{2N}) \left\| u_2^n \bar{u}_1^q D_x^r (u_1-u_2)\right\|_{H_x^{\tilde m +p+2N(p-1), 2N}}^2.
\end{eqnarray}
Now, we can apply Lemma 6.0.1 in \cite{FS} to each term in the right-hand side of \eqref{brutta}. We get
\begin{eqnarray*}
\left\| (u_1-u_2)P_{n-1}(u_1,u_2)\right.&&\hskip-1cm\left. \bar u_1^q D_x^r u_1\right\|_{H_x^{\tilde m +p+2N(p-1), 2N}}^2 \\
&\lesssim&  \| u_1-u_2 \|^2_{H_x^{1/2+\varepsilon, 2N}} \left\|P_{n-1}(u_1,u_2) \bar u_1^q D_x^r u_1\right\|^2_{H_x^{\tilde m +p+2N(p-1)}} \\&&+ \|u_1-u_2 \|^2_{H_x^{\tilde m +p+2N(p-1)}} \left\|P_{n-1}(u_1,u_2) \bar u_1^q D_x^r u_1\right\|^2_{H_x^{1/2+\varepsilon,2N}}.\end{eqnarray*}
Using the algebra properties of the weighted Sobolev spaces, 
$\tilde{m} +p +(2N+1)(p-1) \leq m $ and $\frac{1}{2} + r \leq \frac{1}{2} + p-1 < \tilde{m}$ ($m > \frac{p-1}{2}$), we get 
\begin{eqnarray*}
&&\left\| (u_1-u_2)P_{n-1}(u_1,u_2) \bar u_1^q D_x^r u_1\right\|_{H_x^{\tilde m +p+2N(p-1), 2N}}^2 
\\
&&\lesssim \|u_1-u_2 \|^2_{H_x^{\tilde{m},2N}} \cdot \|u_1\|^{2q}_{H_x^{\tilde{m}+p+2N(p-1)}}\cdot \|u_1 \|^{2}_{H_x^{\tilde m + p+2N(p-1)+ p-1}} \cdot \|P_{n-1}(u_1, u_2)\|^2_{H_x^{\tilde m+p+2N(p-1) }} \\ &&+ \|u_1-u_2 \|^2_{H_x^m} \cdot \|P_{n-1}(u_1, u_2)\|^2_{H_x^{\tilde m,2N}}\|u_1\|_{H_x^{\tilde m,2N}}^{2q} \| u_1\|_{H_x^{\tilde m,2N}}^2\\
&&\lesssim \|u_1\|^{2q+2}_{X_T}(\|u_1\|^{2n-2}_{X_T} + \|u_2\|^{2n-2}_{X_T})  \|u_1-u_2\|_{X_T}^2. 
\end{eqnarray*}
Similarly, for the second term in \eqref{brutta} we obtain 
$$ \left\|u_2^n(\bar u_1-\bar u_2)P_{q-1}(\bar u_1, \bar u_2)D_x^r u_1 \right\|_{H_x^{\tilde m +p+2N(p-1), 2N}}^2 \lesssim \|u_2\|^{2n+2}_{X_T}(\|u_1\|^{2q-2}_{X_T} + \|u_2\|^{2q-2}_{X_T})  \|u_1-u_2\|_{X_T}^2. $$
and for the last one we get
\begin{eqnarray*}
\left\| u_2^n \bar{u}_1^q D_x^r (u_1-u_2)\right\|_{H_x^{\tilde m +p+2N(p-1),2N}}^2 &\lesssim& (\|u_1\|+\|u_2\|)^{2n+2q}_{X_T}\|u_1-u_2\|_{X_T}^2.
\end{eqnarray*}

Coming back to \eqref{square}, we conclude that
\begin{eqnarray} \label{terzastima} \nonumber\| v_1- v_2 \|^2_{L^\infty_t H_x^{\tilde m+p, 2N}} &\lesssim& T (1+T^{2N})\|u_1-u_2\|^2_{X_T}  (\|u_1\|_{X_T}+ \|u_2\|_{X_T})^{2(n+q)}.
\end{eqnarray} 
Gathering the previous estimates \eqref{stimaintegraledellanorma}, \eqref{secondastima} and \eqref{terzastima}, we obtain the following estimate
$$\|v_1-v_2\|^2_{X_T} \lesssim T (1+T^{2N}) (\|u_1\|_{X_T}+ \|u_2\|_{X_T}+ \|g\|_{X_T})^{2n+2q}\|u_1-u_2\|_{X_T}^2.$$
Then, for every fixed $R>0$ we can choose $T$ sufficiently small such that $\Phi$ is a contraction on the closed ball $B_R \subset X_T.$ Then the existence and uniqueness of a solution of the Cauchy problem \eqref{nonlinearCP} in $X_{T^\ast}$ for some $T^\ast \in (0,T]$ follows by standard application of the contraction principle. This concludes the proof of Theorem \ref{mainNL}.
\qed

\medskip

\noindent\textbf{Acknowledgements.} 
The first author was supported by the grant 2022/01712-3 from S\~ao Paulo Research Foundation (FAPESP) and  by the grant 306699/2025-7 from Conselho Nacional de Desenvolvimento Cient\'ificoo e Tecnol\'ogico (CNPq). The second and the third author have been partially supported by the Italian Ministry of the University and Research - MUR, within the PRIN 2022 Call (Project Code
2022HCLAZ8, CUP D53C24003370006).


\vskip1cm
\noindent
$^1$ ALEXANDRE ARIAS JUNIOR, Department of Computing and Mathematics, Universidade de S\~ao Paulo, Ribeir\~ao Preto, Brazil \\
email: alexandre.ariasjunior@usp.br \\

\noindent
$^2$ ALESSIA ASCANELLI, Dipartimento di Matematica e Informatica, University of Ferrara, Via Machiavelli 30, 
44121 Ferrara, Italy \\
email: alessia.ascanelli@unife.it
\\

\noindent
$^3$ MARCO CAPPIELLO, Dipartimento di Matematica ``G. Peano'', University of Turin, Via Carlo Alberto 10, 10123 Torino,
Italy
\\
email: marco.cappiello@unito.it


\begin{thebibliography}{AAA}
 
\bibitem{ABZsuff}
A. Ascanelli, C. Boiti, L. Zanghirati, {\it Well-posedness of the Cauchy problem for p-evolution equations}, J. Differential Equations \textbf{253} (2012) n. 10, 2765-2795.

 \bibitem{ABZnec}
A. Ascanelli, C. Boiti, L. Zanghirati, {\it A Necessary condition for $ H^{\infty}$ well-posedness of $ p $-evolution equations}, Adv.  Differential Equations \textbf{21} (2016), 1165-1196.

\bibitem{ABmoser}A.Ascanelli, C.Boiti: {\it Semilinear p-evolution equations in Sobolev spaces}, J. Differential Equations 260 (2016), 7563-7605

\bibitem{AscanelliCappielloLL2006}
A. Ascanelli, M. Cappiello, \textit{Log-Lipschitz regularity
	for SG hyperbolic systems}, J. Differential Equations  \textbf{230} (2006), 556-578.

\bibitem{Cai}
H. Cai, \textit{Dispersive smoothing effects for KdV type equations}, J. Differential Equations \textbf{136} (1997), n. 2, 191-221.

\bibitem{chihara}
H. Chihara, \textit{Smoothing effects of dispersive pseudodifferential equations}, Comm. Partial Differential Equations \textbf{27} (2002) n. 9-10, 1953-2005.

\bibitem{CC} M.~Cicognani, F.~ Colombini, {\sl The Cauchy 
problem for $p$-evolution equations.},
Trans. Amer. Math. Soc. {\bf 362}, n. 9 (2010), 4853-4869.

\bibitem{CR} 
M. Cicognani, M.  Reissig, {\it Well-posedness for degenerate Schr\"odinger equations}, Evol. Equ Control Th. \textbf{3} (1) (2014), 15-33.

\bibitem{ConstantinSaut1}
P. Constantin, J.-C. Saut, \textit{Local smoothing properties of dispersive equations}, J. Amer. Mah. Soc. \textbf{1} (1988), 413-439.

\bibitem{ConstantinSaut2}
P. Constantin, J.-C. Saut, \textit{Local smoothing properties of Schr\"odinger equations}, Indiana Univ. Math. J. \textbf{38} (1989), 791-810.

\bibitem{CKS}
W. Craig, T. Kappeler, W. Strauss, \textit{Microlocal dispersive smoothing for the Schr\"odinger equation}, Comm. Pure Appl. Math. \textbf{48} (1995), 769-860.

\bibitem{Cordes}
H.O. Cordes, \textit{The technique of pseudodifferential operators}, Cambridge Univ. Press, 1995.


\bibitem{Doi1}
S. Doi, {\it On the Cauchy problem for Schr\"odinger type equations and the regularity of solutions}, J. Math. Kyoto Univ. \textbf{34} (1994), 319-328.

\bibitem{Doi2}
S. Doi, {\it Remarks on the Cauchy problem for Schr\"odinger-type equations}, Comm. Partial Differential Equations \textbf{21} (1996), 163-178.

\bibitem{FS}
S. Federico, G. Staffilani, \textit{Smoothing effect for time-degenerate Schr\"odinger operators}, J. Differential Equations \textbf{298} (2021), 205-247.

\bibitem{FT}
S. Federico, D. Tramontana, \textit{Smoothing effect for third order operators with variable coefficients}, Preprint 2025, https://arxiv.org/abs/2503.08656

\bibitem{FP}
C. Fefferman, D.H. Phong, \textit{On positivity of pseudo-differential operators}, Proc. Natl. Acad. Sci. USA 75 (10) (1978) 4673-4674.

\bibitem{Gomez}
C.A. Gomez, \'A.S. Salas, S. Casanova Trujillo, \textit{Forced variable-coefficient Kawahara equations: closed-form traveling waves via tanh-coth and elliptic ans\"atze}, Front. Phys., Sec. Statistical and Computational Physics \textbf{14} (2026) https://doi.org/10.3389/fphy.2026.1807760

\bibitem{I1}
W. Ichinose, {\sl Some remarks on the Cauchy problem for Schr{\"o}dinger type equations.}
 Osaka J.  Math. \textbf{21} (3) (1984) 565-581.


\bibitem{I2} W.~Ichinose, {\sl Sufficient condition on $H^{\infty }$
well-posedness for Schr\"odinger type equations},
Comm. Partial Differential Equations, {\bf 9}, n.1 (1984), 33-48.

\bibitem{Israwi}
S. Israwi, \textit{Higher order dissipative-dispersive system and application}, J. Partial Differential Eqs. \textbf{33} (2020) n. 3, 1-16.

\bibitem{KB}
K. Kajitani, A. Baba, {\sl The Cauchy problem for Schr{\"o}dinger type equations.}
Bull. Sci. Math.\textbf{119} (5), (1995), 459-473.
 
\bibitem{Kato}
T. Kato, \textit{On the Cauchy problem for the (generalized) Korteweg-de Vries equation},
Adv. Math. Suppl. Stud., Stud. Appl. Math.
\textbf{8} (1983), 93-128.

 
\bibitem{KPV1}
C. Kenig, G. Ponce, L. Vega, \textit{Oscillatory integrals and regularity of dispersive equations}, Indiana Univ. Math. J. \textbf{40} (1991) n. 1, 33-69.

\bibitem{KPV2}
C.E. Kenig, G. Ponce, L. Vega, \textit{Small solutions to nonlinear
Schr\"odinger Equations}, Ann. Inst. H. Poincar\'e Anal. Non. Lineaire
 \textbf{10}  1993, 255-288.
 
\bibitem{KPV3} 
C.E. Kenig, G. Ponce, L. Vega, \textit{Well-posedness and scattering
results for the generalized Korteweg-de Vries equation via the
contraction principle}, Comm. Pure Appl. Math. \textbf{46} (1993), 527-620.

\bibitem{KPV4} 
C.E. Kenig, G. Ponce, L. Vega, \textit{On the generalized Benjamin-Ono
equation}, Trans. Amer. Math. Soc. \textbf{342} (1994), 155-172.

\bibitem{KPV5}
C. Kenig, G. Ponce, L. Vega, \textit{Smoothing effects and local existence theory for the generalized nonlinear Schr\"odinger equations}, Invent. Math. \textbf{134} (1988), 489-545.

\bibitem{KPRV}
C. Kenig, G. Ponce, C. Rolvung, L. Vega, \textit{Variable coefficients Schr\"odinger flows and ultrahyperbolic operators}, Adv. Math. \textbf{196} (2005), 373-486.

\bibitem{KG}H. Kumano-Go, {\it Pseudo-Differential Operators}, MIT Press, Cambridge, London, 1982.

\bibitem{KW}
K. Kajitani, S. Wakabayashi, \textit{Microhyperbolic operators in Gevrey classes},  Publ. Res. Inst. Math. Sci. {\bf 25} (1989) n. 2, 169-221. 

\bibitem{Mizohata1961}
S. Mizohata, {\it Some remarks on the Cauchy problem}, J. Math. Kyoto Univ. \textbf{1} (1961) n. 1, 109-127.

\bibitem{mizohata2014}
S. Mizohata, {\it On the Cauchy problem}, Academic Press. Volume 3 (2014).

\bibitem{Molinet}
L. Molinet, R. Talhouk, I. Zaiter, \textit{On well-posedness for some Korteweg-de Vries type equations with variable coefficients}, J. Evol. Equ. 23 (2023), no. 3, Paper No. 52, 37 pp.

\bibitem{Parenti}
C. Parenti, \textit{Operatori pseudodifferenziali in $\mathbb{R}^n$ e applicazioni}, Ann.
Mat. Pura Appl. \textbf{93}, 359--389 (1972).

\bibitem{Schrohe}
E. Schrohe, \textit{Spaces of weighted symbols and weighted Sobolev spaces on manifolds.}
In ``Pseudodifferential Operators'', Proceedings Oberwolfach 1986. H. O. Cordes, B. Gramsch and H. Widom editors, Springer LNM, \textbf{1256} New York, 360--377  (1987).


\bibitem{Sugimoto}
M. Sugimoto, \textit{Global smoothing properties of generalized Schr\"odinger equations}, J. Anal. Math. \textbf{76} (1998), 191-204. 

\end{thebibliography}
\end{document}